\documentclass[pdflatex,sn-mathphys-num]{sn-jnl}
\usepackage{graphicx}
\usepackage{amsmath,amssymb,amsfonts}
\usepackage{amsthm}
\usepackage{mathtools}
\usepackage{bm}
\usepackage{booktabs}
\usepackage{enumitem}
\usepackage[title]{appendix}

\theoremstyle{thmstyleone}
\newtheorem{theorem}{Theorem}[section]
\newtheorem{proposition}[theorem]{Proposition}
\newtheorem{lemma}[theorem]{Lemma}
\newtheorem{corollary}[theorem]{Corollary}
\theoremstyle{thmstyletwo}
\newtheorem{remark}[theorem]{Remark}
\theoremstyle{thmstylethree}

\numberwithin{equation}{section}
\newcommand{\R}{\mathbb R}
\newcommand{\Pp}{\mathcal P}
\newcommand{\tri}{\triangle}
\newcommand{\Dpts}{\mathcal D}
\newcommand{\Ncal}{\mathcal N}
\newcommand{\RT}{\mathcal{RT}}
\newcommand{\dd}{\,\mathrm d}
\newcommand{\curl}{\operatorname{curl}}
\newcommand{\divv}{\operatorname{div}}
\newcommand{\rot}{\operatorname{rot}}
\newcommand{\range}{\operatorname{range}}
\newcommand{\rank}{\operatorname{rank}}
\newcommand{\diag}{\operatorname{diag}}

\newcommand{\dist}{\operatorname{dist}}
\newcommand{\norm}[1]{\lVert #1\rVert}
\newcommand{\abs}[1]{\lvert #1\rvert}
\newcommand{\eps}{\varepsilon}
\newcommand{\wideT}{\widehat T}

\begin{document}

\title[Bernstein Constraint Complexes]{Bernstein Constraint Complexes for Multivariate Splines on Triangulated Surfaces}

\author*[1]{\fnm{Shelvean} \sur{Kapita}}\email{kapita@tamu.edu}

\affil*[1]{\orgdiv{Department of Mathematics}, \orgname{Texas A\&M University}, \orgaddress{\street{155 Ireland Street}, \city{College Station}, \postcode{77843}, \state{TX}, \country{USA}}}

\abstract{We introduce Bernstein constraint complexes, a representation of finite element differential complexes in which every triangle of a triangulation keeps its own Bernstein--B\'ezier coefficients and global continuity is imposed only through smoothness functionals attached to edges and vertices. Continuity of scalar splines, tangential continuity of spline vector fields, normal continuity of spline vector fields, and higher smoothness differ only in the edge functionals. No global basis is constructed. The structural result is a coefficient-level commuting identity: on every edge the smoothness functionals of a Bernstein derivative are explicit combinations of the smoothness functionals of its argument, through univariate Bernstein difference matrices. This gives a local proof that the kernels of the smoothness matrices form a subcomplex, for the de Rham traces and for componentwise $C^r$ spline profiles, and a computable row-space test for candidate smooth profiles. We realize the lowest-order Powell--Sabin exact sequence in this form and confirm its exactness by ranks. The constrained Galerkin problem is solved in the broken coefficients with exact smoothness equations; we prove equivalence with the conforming method, treat semidefinite element operators through a pseudoinverse interface system, and give a singular pencil whose finite spectrum is the constrained Maxwell spectrum. Under the covariant and contravariant surface Piola maps the reference functionals of the de Rham sequence transfer unchanged to conforming curved triangulations. Experiments compare the null-space method, the augmented Lagrangian iteration of Awanou, Lai and Wenston, and direct solution of the saddle system on the same problems, and cover planar source and eigenvalue problems, a $C^1$ biharmonic solve, the exact sphere, and an embedded hyperboloid.}



\maketitle

\section{Introduction}\label{sec:intro}

Polynomial splines on triangulations allow the degree and the smoothness across edges to be prescribed independently, and every smoothness condition can be written explicitly as a linear relation among Bernstein--B\'ezier coefficients \cite{LaiSchumaker2007}. For scalar elliptic, plate, obstacle and surface problems this has produced a family of solvers that work directly with the B-coefficients \cite{AwanouLaiWenston2006,Lai2025}. For vector partial differential equations the correct global space is often not a componentwise scalar spline space. Maxwell equations require tangential continuity, mixed diffusion and flux formulations require normal continuity, and finite element exterior calculus organizes these spaces into differential complexes whose discrete operators inherit the identities of the de Rham complex \cite{ArnoldFalkWinther2006,ArnoldFalkWinther2010}.

The local spaces are classical. N\'ed\'elec, Raviart--Thomas and BDM families provide the $H(\curl)$ and $H(\divv)$ elements \cite{Nedelec1980,Nedelec1986,RaviartThomas1977,BrezziDouglasMarini1985}, Bernstein--B\'ezier bases for Raviart--Thomas elements of arbitrary order were constructed by Ainsworth, Andriamaro and Davydov \cite{AinsworthAndriamaroDavydov2015}, and Ainsworth and Fu built Bernstein--B\'ezier bases for the full tetrahedral de Rham sequence \cite{AinsworthFu2018}. Nodal realizations of high-order edge and face spaces treat tangential and normal continuity as directional continuity of vector polynomial data \cite{ChenChenHuangWei2024}, and smooth complexes on Powell--Sabin, Alfeld, Clough--Tocher and Worsey--Farin refinements come with commuting projections and exactness results adapted to the added smoothness \cite{ChristiansenHu2018,GuzmanLischkeNeilan2020,GuzmanLischkeNeilan2022}. Partially discontinuous nodal elements give a further route between continuous Lagrange fields and minimally conforming edge and face fields \cite{HuHuZhang2022}.

On the spline side, Lai and Schumaker developed the Bernstein--B\'ezier theory of spline spaces on triangulations, including the explicit smoothness conditions and the theory of determining sets \cite{LaiSchumaker2007}. Alfeld and Sorokina studied the images and kernels of the gradient, curl, divergence and Laplace operators on bivariate spline spaces and spline vector fields \cite{AlfeldSorokina2016}, and Sorokina developed Bernstein--B\'ezier techniques for the divergence of spline vector fields in higher dimension \cite{Sorokina2018}. Closest to the present work is the constrained B-form method of Awanou, Lai and Wenston \cite{AwanouLaiWenston2006}. They start from the discontinuous spline space, write smoothness and boundary data as linear equations on the B-coefficients, impose those equations by Lagrange multipliers, and solve the saddle system by an augmented Lagrangian matrix iteration whose rate is analyzed in \cite{AwanouLai2005}. The same constrained-coefficient framework underlies the spline collocation method of Lai and Lee \cite{LaiLee2022,LaiLee2023}, in which the strong form of the equation is imposed at collocation points inside the triangles or tetrahedra, the smoothness conditions are imposed as constraints on the B-coefficients, and the resulting overdetermined constrained system is solved by least squares; there the smoothness matrix plays exactly the role it plays here, and the constraint complexes of this paper apply to it unchanged. The spaces used in this paper are classical and the constrained B-form system is theirs; what is new is the complex formed by the smoothness matrices, the explicit factors $B_k$ through which differentiation acts on them, the row-space test, the constrained Maxwell pencil, and the invariance of the reference functionals under the surface Piola maps.

This paper begins one level above that construction. The question is whether the smoothness matrices themselves can be organized as the consecutive spaces of a differential complex, with scalar continuity, tangential continuity, normal continuity and higher-order smoothness written in the same edge language, differentiation acting on the broken coefficients, and compatibility read off the smoothness functionals alone. The viewpoint is this: keep the B-coefficients broken by triangle, encode continuity by smoothness functionals attached to edges, and let the differential operators act on the broken coefficients. Each global space is then a kernel
\[
   V_h^k=\ker C_k
\]
of a sparse smoothness matrix $C_k$, and a basis of that kernel is never required. We call the resulting sequence of kernels a Bernstein constraint complex.

The common algebraic object is the constrained Galerkin system. In one slot of the complex let $c_k$ be the broken B-coefficient vector, let $A_k$ and $f_k$ be the broken Galerkin matrix and load vector, and let $C_kc_k=g_k$ collect the smoothness and boundary equations, whose rows are allowed to be redundant. The stationarity equations are
\begin{equation}\label{eq:intro-kkt}
   A_kc_k+C_k^T\lambda_k=f_k,\qquad C_kc_k=g_k .
\end{equation}
Neither the formulation nor the uniqueness of $c_k$ requires the broken matrix $A_k$ to be invertible; redundant rows of the smoothness matrix appear as nonuniqueness of the multiplier. There are three ways to solve \eqref{eq:intro-kkt} on the same space. One may build a sparse null-space matrix $Z_k$ with $\range Z_k=\ker C_k$ and solve the reduced system, which is the null-space method of \cite[Sec.~6]{BenziGolubLiesen2005}; in spline terms $Z_k$ extends the coefficients on a determining set to all coefficients \cite[Chap.~5]{LaiSchumaker2007}. One may keep the broken coefficients and run the augmented Lagrangian iteration of Awanou, Lai and Wenston, whose eliminated matrix is $A_k+\eps^{-1}C_k^TC_k$ \cite[Algorithm~5]{AwanouLaiWenston2006}. Or one may keep the broken coefficients and solve the saddle system \eqref{eq:intro-kkt} directly. These are three coordinate systems for one problem, and Section~\ref{sec:numerics} compares them on identical data before the remaining experiments settle on one.

The central observation is that differential compatibility can be read off the smoothness functionals. Let $C_{0,e}$ be the jump of the degree-$(d+1)$ scalar B-coefficients across an edge $e$ and let $C_{1,e}^{\curl}$ be the jump of the degree-$d$ B-coefficients of the tangential component of a vector field. If $D_0^{\rm br}$ is the broken Bernstein gradient, then
\begin{equation}\label{eq:intro-commute}
   C_{1,e}^{\curl}D_0^{\rm br}=|e|^{-1}\mathsf U_{d+1}\,C_{0,e},
\end{equation}
where $\mathsf U_{d+1}$ is the univariate Bernstein difference matrix from degree $d+1$ to degree $d$ on the edge. Hence $C_0c=0$ implies $C_1^{\curl}D_0^{\rm br}c=0$ without any global basis. The rotated identity gives the normally continuous complex. More generally, for smoothness matrices $J_k$ of any order, a candidate complex is compatible exactly when $J_{k+1}D_k=B_kJ_k$ for edge-local matrices $B_k$, and for componentwise $C^r$ profiles we compute $B_k$ explicitly.

The second contribution concerns curved triangulations. For a surface element map $F_T:\wideT\to T$, the covariant surface Piola map preserves tangential line traces and the contravariant map preserves co-normal fluxes. The reference smoothness functionals for $H^1$, $H(\curl)$ and $H(\divv)$ are therefore inherited without change on a conforming curved triangulation, and curvature enters only the element mass and stiffness matrices. Higher-order smoothness across curved patches depends on derivatives of the geometry map and is treated only on planar meshes.

Maxwell eigenvalues are the most demanding test used here, since spurious modes reveal defects in the gradient--curl structure that coercive source problems hide \cite{BoffiGuzmanNeilan2023}. We solve the eigenproblem in the broken coefficients through a generalized pencil with a singular mass block. Its finite eigenvalues are exactly the eigenvalues on $\ker C_1^{\curl}$, and the multiplier block belongs to the infinite part of the pencil, so the physical spectrum is targeted by shift-invert without a penalty parameter and without a conforming basis.

\paragraph{Contributions.}
\begin{enumerate}[leftmargin=2em]
\item Scalar, tangential, normal and $L^2$ continuity are organized as consecutive kernels of one edge-local smoothness complex, for full polynomial and for trimmed N\'ed\'elec and Raviart--Thomas families (Section~\ref{sec:constraints}).
\item The edge commuting identity \eqref{eq:intro-commute} and its consequences: a local proof of the subcomplex property and exactness on simply connected domains (Section~\ref{sec:commuting}).
\item A row-space criterion for arbitrary smoothness profiles, an explicit factorization for componentwise $C^r$ profiles, and the realization of the lowest-order Powell--Sabin exact sequence with the classical $C^1$ quadratic spline space as its first term (Section~\ref{sec:smooth}).
\item Equivalence of the constrained system with the conforming Galerkin method, the null-space and augmented Lagrangian realizations of the same equations, an interface system valid for semidefinite element matrices, and a constrained Maxwell pencil (Section~\ref{sec:variational}).
\item Transfer of the reference smoothness functionals to conforming curved triangulations through the surface Piola maps (Section~\ref{sec:surface}).
\item Numerical experiments comparing the three realizations, and validating planar source and eigenvalue problems, the Powell--Sabin complex and biharmonic problem, and curved problems on the exact sphere and an embedded hyperboloid (Section~\ref{sec:numerics}).
\end{enumerate}

\paragraph{Scope.}
The local spaces used here, the N\'ed\'elec, BDM, Raviart--Thomas and Powell--Sabin spaces, are classical, and the constrained B-form system with Lagrange multipliers is that of \cite{AwanouLaiWenston2006}. What is new is the organization of scalar, vector and higher-order smoothness conditions as one edge-local complex, the explicit commuting factors, the row-space criterion and its use for smooth profiles, the constrained Maxwell pencil, and the separation of de Rham smoothness functionals from curved-surface metric assembly. The subcomplex results are algebraic; exactness of smooth sequences and bounded commuting projections are properties of the chosen spaces and are taken from the literature where needed.

\paragraph{Conventions.}
For the de Rham sequences we index by the degree $d$ of the vector field, as is customary for N\'ed\'elec spaces, so the scalar potentials have degree $d+1$ and the terminal space has degree $d-1$. For scalar smooth profiles in Section~\ref{sec:smooth} we index by the scalar degree, as in \cite{AlfeldSorokina2016}. American spelling is used throughout.

Section~\ref{sec:bb} recalls splines in B-form. Section~\ref{sec:constraints} defines the constraint spaces. Section~\ref{sec:commuting} proves the commuting identities and exactness. Section~\ref{sec:smooth} treats higher smoothness and the Powell--Sabin complex. Section~\ref{sec:variational} gives the constrained Galerkin and eigenvalue formulations. Section~\ref{sec:surface} transfers the construction to curved surfaces. Section~\ref{sec:analysis} records approximation and spectral consequences, and Section~\ref{sec:numerics} reports the experiments.

\section{Splines in B-form on triangulations}\label{sec:bb}

\subsection{Domain points and B-coefficients}

Let $\tri$ be a regular triangulation of a polygonal domain $\Omega\subset\R^2$, with $N_V$ vertices, $N_E$ edges and $N_T$ triangles. For a triangle $T=\langle v_1,v_2,v_3\rangle$ with barycentric coordinates $\lambda_1,\lambda_2,\lambda_3$, the Bernstein polynomials of degree $d$ are
\begin{equation}\label{eq:Bernstein}
   B^d_{ijk}=\frac{d!}{i!\,j!\,k!}\lambda_1^i\lambda_2^j\lambda_3^k,\qquad i+j+k=d,
\end{equation}
and every $p\in\Pp_d(T)$ has a unique B-form $p=\sum_{i+j+k=d}c_{ijk}B^d_{ijk}$. Following \cite{LaiSchumaker2007}, the B-coefficient $c_{ijk}$ is associated with the domain point $\xi_{ijk}=(iv_1+jv_2+kv_3)/d$, and $\Dpts_{d,T}$ denotes the set of these $\binom{d+2}{2}$ domain points. We write $c_T$ for the vector of B-coefficients of $p$ on $T$, ordered by domain points. For a vector polynomial $p=(p^x,p^y)\in[\Pp_d(T)]^2$ the coefficient vector consists of two copies, $c_T=(c^x_T,c^y_T)$.

The spline spaces used below are
\begin{equation}\label{eq:Srd}
   S^r_d(\tri)=\{s\in C^r(\Omega):\ s|_T\in\Pp_d(T)\ \ \forall T\in\tri\},\qquad r\ge0,
\end{equation}
and the discontinuous space $S^{-1}_d(\tri)$ of piecewise polynomials of degree $d$ with no continuity across edges. A spline in $S^{-1}_d(\tri)$ is determined by the concatenation $c\in\R^{N}$, $N=N_T\binom{d+2}{2}$, of its B-coefficient vectors on the triangles of $\tri$. No coefficient is shared between triangles. Spline vector fields are pairs of splines, $[S^{-1}_d(\tri)]^2$, with coefficient vectors of length $2N$; this is the setting of \cite{AlfeldSorokina2016}.

\subsection{Directional derivatives and edge restrictions}

For a direction $z\in\R^2$ let $D_z$ denote the directional derivative. Writing $a_m=D_z\lambda_m$ for the directional coordinates of $z$ relative to $T$,
\begin{equation}\label{eq:bern-deriv}
   D_zB^d_{ijk}=d\bigl(a_1B^{d-1}_{i-1,j,k}+a_2B^{d-1}_{i,j-1,k}+a_3B^{d-1}_{i,j,k-1}\bigr),
\end{equation}
with the convention that a Bernstein polynomial with a negative index is zero \cite[Thm.~2.8]{LaiSchumaker2007}. Hence $D_z$ maps the B-coefficients of degree $d$ to those of degree $d-1$ by a sparse matrix in which each row has at most three nonzero entries. We write $G_T$ for the matrix of the gradient on $T$, from $\Pp_{d+1}(T)$ to $[\Pp_d(T)]^2$, and $D_0^{\rm br}=\diag_TG_T$ for the broken gradient on $S^{-1}_{d+1}(\tri)$.

Let $e=\langle v_1,v_2\rangle$ be the edge of $T$ opposite $v_3$. Restriction to $e$ sets $\lambda_3=0$, so
\begin{equation}\label{eq:edge-restriction}
   p|_e=\sum_{i+j=d}c_{ij0}B^d_{ij0}|_e,
\end{equation}
and the B-coefficients of $p|_e$ are the B-coefficients of $p$ at the $d+1$ domain points on $e$ \cite[Lemma~2.6]{LaiSchumaker2007}. Ordering them from $v_1$ to $v_2$ identifies $p|_e$ with a univariate polynomial in the Bernstein basis $b^d_j(t)=\binom dj t^j(1-t)^{d-j}$, $0\le t\le1$. For a univariate coefficient vector $r=(r_0,\dots,r_d)$,
\begin{equation}\label{eq:univariate-derivative}
   \frac{\dd}{\dd t}\sum_{j=0}^dr_jb^d_j(t)=d\sum_{j=0}^{d-1}(r_{j+1}-r_j)b^{d-1}_j(t),
\end{equation}
and we denote the corresponding univariate difference matrix by
\begin{equation}\label{eq:Delta}
   \mathsf U_d=d
   \begin{bmatrix}
   -1&1&&\\
   &-1&1&\\
   &&\ddots&\ddots\\
   &&&-1&1
   \end{bmatrix}\in\R^{d\times(d+1)} .
\end{equation}
With the unit tangent $t_e$ pointing from $v_1$ to $v_2$, the tangential derivative $D_{t_e}p|_e$ has B-coefficients $|e|^{-1}\mathsf U_dr$, where $r$ is the coefficient vector of $p|_e$.

\subsection{Smoothness conditions across an edge}

Let $T=\langle v_1,v_2,v_3\rangle$ and $\widetilde T=\langle v_4,v_3,v_2\rangle$ share the edge $e=\langle v_2,v_3\rangle$, and let $s\in S^{-1}_d(\tri)$ have B-coefficients $c_{ijk}$ on $T$ and $\widetilde c_{ijk}$ on $\widetilde T$, each indexed by the vertex order of its own triangle. By \cite[Thm.~2.28]{LaiSchumaker2007}, $s$ is $C^r$ across $e$ if and only if
\begin{equation}\label{eq:LS-smoothness}
   \widetilde c_{njk}=\sum_{\nu+\mu+\kappa=n}c_{\nu,\,k+\mu,\,j+\kappa}\,B^n_{\nu\mu\kappa}(v_4),
   \qquad j+k=d-n,\quad n=0,\dots,r,
\end{equation}
where $B^n_{\nu\mu\kappa}(v_4)$ are the Bernstein polynomials of degree $n$ relative to $T$ evaluated at $v_4$. Each condition involves the coefficients in the rows $n$ and below parallel to $e$ on both triangles, and the difference of the two sides of \eqref{eq:LS-smoothness} is a linear functional of the broken coefficient vector. We call these the smoothness functionals of order $r$ across $e$ and write $J^{(r)}_{0,e}$ for the matrix whose rows are these functionals, so that $s$ is $C^r$ across $e$ if and only if $J^{(r)}_{0,e}c=0$. For $r=0$ the functionals are the differences of the coefficients at the shared domain points on $e$, and we write $C_{0,e}=J^{(0)}_{0,e}$. Stacking the blocks over all interior edges gives the smoothness matrix $J^{(r)}_0$ of the space, and
\begin{equation}\label{eq:Srd-kernel}
   S^r_d(\tri)=\ker J^{(r)}_0\subset S^{-1}_d(\tri).
\end{equation}
The rows of $J^{(r)}_0$ are in general linearly dependent, since the conditions around a vertex are related; this is the source of the dimension theory in \cite[Chap.~9]{LaiSchumaker2007} and it is allowed throughout this paper.

An equivalent description will be convenient in Section~\ref{sec:smooth}. Fix a unit vector $n_e$ transverse to $e$. Then $s$ is $C^r$ across $e$ if and only if the univariate polynomials $(D_{n_e}^ms)|_e$, $m=0,\dots,r$, agree from the two sides; the row space of the corresponding functionals equals the row space of \eqref{eq:LS-smoothness}, since both characterize the same kernel. We write $\gamma_m(s)\in\R^{d+1-m}$ for the B-coefficient vector of $(D_{n_e}^ms|_T)|_e$, which is a sparse linear function of the first $m+1$ rows of coefficients parallel to $e$ by \eqref{eq:bern-deriv} and \eqref{eq:edge-restriction}.

\subsection{Local polynomial de Rham sequences}

For $d\ge1$ the full polynomial sequence on a triangle is
\begin{equation}\label{eq:local-full}
   \R\hookrightarrow\Pp_{d+1}(T)\xrightarrow{\nabla}[\Pp_d(T)]^2\xrightarrow{\curl}\Pp_{d-1}(T)\to0,
   \qquad \curl v=\partial_xv^y-\partial_yv^x .
\end{equation}
It is exact: $\ker\curl$ consists of gradients of polynomials of degree at most $d+1$ because $T$ is simply connected, and $\curl$ is onto $\Pp_{d-1}(T)$ since for $q\in\Pp_{d-1}(T)$ the field $v=(0,\int_0^xq(\sigma,y)\dd\sigma)$ lies in $[\Pp_d(T)]^2$ and has $\curl v=q$. With the rotation $R(a,b)=(-b,a)$, the sequence $\Pp_{d+1}\xrightarrow{R\nabla}[\Pp_d]^2\xrightarrow{\divv}\Pp_{d-1}$ is exact as well. After tangential assembly the middle space of \eqref{eq:local-full} is the second N\'ed\'elec family, and after normal assembly of the rotated sequence it is the BDM family.

The construction below applies to any choice of local space. The first N\'ed\'elec space $\Ncal_d(T)=[\Pp_d(T)]^2\oplus x^\perp\widetilde{\Pp}_d(T)$, of order $d+1$ in the usual numbering and with tangential traces of degree $d$, with $x^\perp=(-y,x)$ and $\widetilde{\Pp}_d$ the homogeneous polynomials of degree $d$, and its rotation, the Raviart--Thomas space $\RT_d(T)=[\Pp_d(T)]^2\oplus x\widetilde{\Pp}_d(T)$, have edge traces that are univariate polynomials of degree $d$ and can be written in the univariate Bernstein basis. The local sequences $\Pp_{d+1}\xrightarrow{\nabla}\Ncal_d\xrightarrow{\curl}\Pp_d$ and $\Pp_{d+1}\xrightarrow{R\nabla}\RT_d\xrightarrow{\divv}\Pp_d$ are exact, and the smoothness functionals below apply after replacing the local coefficient-to-edge map by the one for the chosen basis. The experiments use the full family because it keeps the algebra transparent.

\section{Constraint spaces on a triangulation}\label{sec:constraints}

\subsection{Tangentially and normally continuous spline vector fields}

Each interior edge $e=T^+\cap T^-$ receives a fixed orientation with unit tangent $t_e$ and unit normal $n_e=Rt_e$. For $v\in[S^{-1}_d(\tri)]^2$ with coefficients $(c^x_T,c^y_T)$ on $T$, the tangential component $v\cdot t_e$ restricted to $e$ from the side of $T$ is a univariate polynomial of degree $d$ whose B-coefficients are $(t_e)_xc^x_\xi+(t_e)_yc^y_\xi$ at the domain points $\xi$ on $e$, by \eqref{eq:edge-restriction}. The normal component is described in the same way with $n_e$. The tangential and normal jump functionals across $e$ are the differences of these coefficient vectors from the two sides, and we write
\begin{equation}\label{eq:C1e}
   C_{1,e}^{\curl}c=\bigl(\text{coefficients of }v|_{T^+}\cdot t_e\bigr)-\bigl(\text{coefficients of }v|_{T^-}\cdot t_e\bigr),
\end{equation}
and $C_{1,e}^{\divv}$ for the normal analogue. Each is a $(d+1)\times2N$ sparse matrix with $2(d+1)$ nonzero entries per row. Boundary conditions are represented by one-sided versions of the same functionals on boundary edges. Figure~\ref{fig:constraints} shows the three kinds of edge functional on a pair of triangles.

\begin{figure}[htbp]
\centering
\includegraphics[width=.8\textwidth]{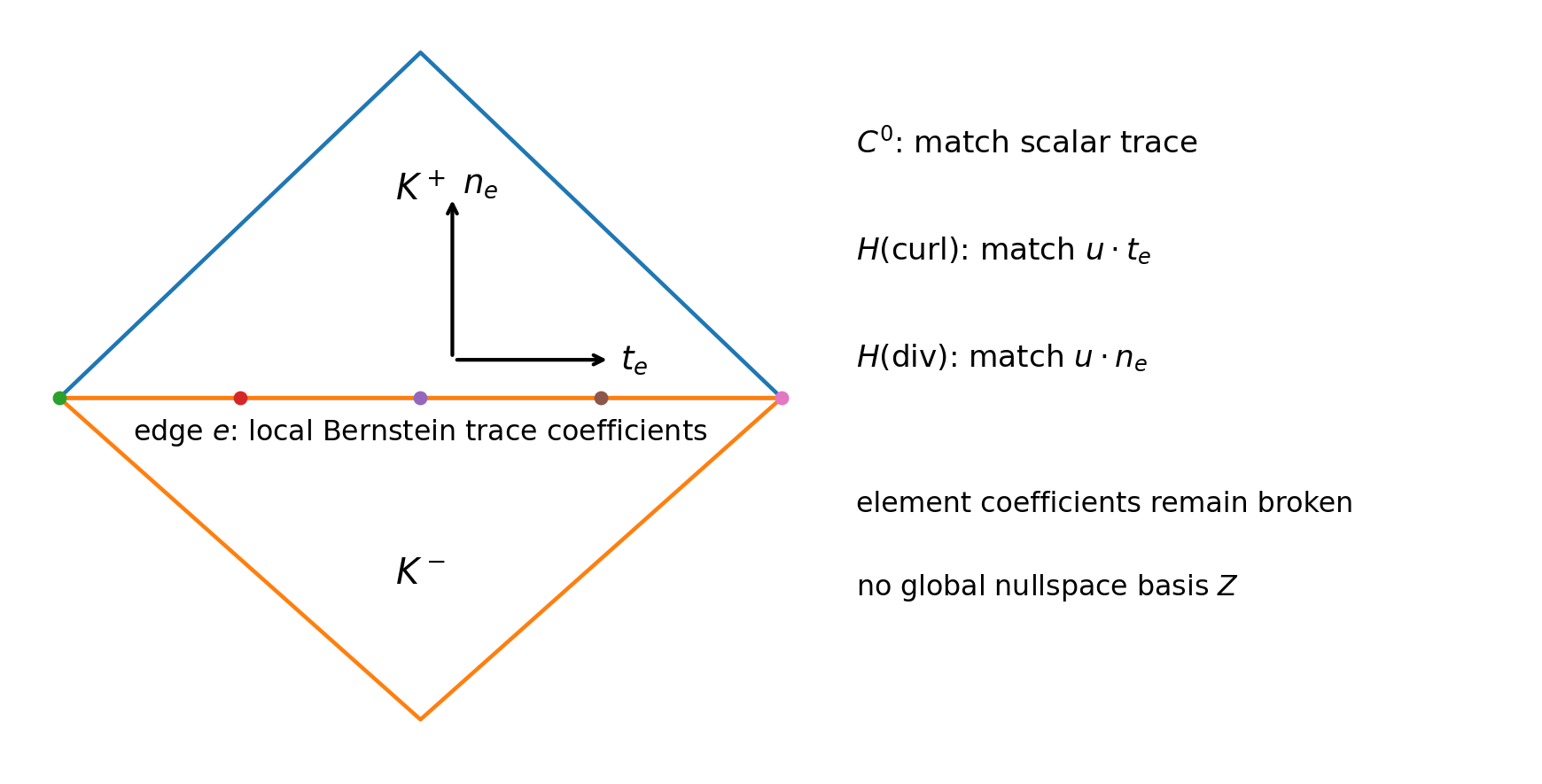}
\caption{Smoothness functionals on a pair of triangles. The B-coefficients remain broken. Continuity of a scalar spline matches the coefficients at the shared domain points; tangential continuity of a spline vector field matches the coefficients of the tangential component on the edge; normal continuity matches those of the normal component. No coefficient is identified across triangles.}
\label{fig:constraints}
\end{figure}

Stacking the interior-edge blocks gives the smoothness matrices $C_0$, $C_1^{\curl}$ and $C_1^{\divv}$, and we define
\begin{equation}\label{eq:spaces}
   V_h^0=\ker C_0,\qquad V_{h,\curl}^1=\ker C_1^{\curl},\qquad V_{h,\divv}^1=\ker C_1^{\divv},\qquad V_h^2=S^{-1}_{d-1}(\tri),
\end{equation}
with $C_0$ acting on $S^{-1}_{d+1}(\tri)$ and the two vector matrices on $[S^{-1}_d(\tri)]^2$. Homogeneous essential boundary conditions are imposed by appending the one-sided boundary blocks. The stacked matrices are a notation; an implementation stores the edge blocks and assembles from edge--triangle incidences.

\begin{theorem}[identification of the kernels]\label{thm:conformity}
The spaces \eqref{eq:spaces} satisfy
\begin{align*}
 V_h^0&=S^0_{d+1}(\tri)=\{u\in H^1(\Omega):u|_T\in\Pp_{d+1}(T)\},\\
 V_{h,\curl}^1&=\{v\in H(\curl;\Omega):v|_T\in[\Pp_d(T)]^2\},\\
 V_{h,\divv}^1&=\{v\in H(\divv;\Omega):v|_T\in[\Pp_d(T)]^2\},\\
 V_h^2&=\{q\in L^2(\Omega):q|_T\in\Pp_{d-1}(T)\}.
\end{align*}
\end{theorem}
\begin{proof}
A piecewise polynomial is in $H^1(\Omega)$ if and only if its two restrictions agree on every interior edge, and by \eqref{eq:edge-restriction} and the linear independence of the univariate Bernstein basis this is equality of the coefficients at the shared domain points, that is $C_0c=0$. A piecewise smooth vector field is in $H(\curl;\Omega)$ if and only if its tangential component is single valued across every interior edge; the tangential component is a polynomial of degree $d$ on $e$ whose coefficients are given after \eqref{eq:C1e}, so single-valuedness is $C_{1,e}^{\curl}c=0$ for every $e$. The normal statement is identical, and $L^2$ imposes no condition across edges.
\end{proof}

The theorem identifies classical spaces and creates no new element. The point is the representation: the conforming space is reached through the smoothness matrix instead of through a global basis, and the same storage holds scalar $C^r$ conditions, smooth macro-element complexes, and changing geometry maps.

\subsection{Dimension from the rank}

The dimension of each space is the dimension of the broken space minus the rank of its smoothness matrix, whether or not the rows are independent. For the full degree-$d$ tangentially continuous space every interior edge contributes $d+1$ functionals and $[\Pp_d(T)]^2$ has dimension $(d+1)(d+2)$, so with $d+1$ shared coefficients per edge and $(d+1)(d-1)$ interior ones per triangle,
\begin{equation}\label{eq:dimV1}
   \dim V_{h,\curl}^1=(d+1)N_E+(d^2-1)N_T,
\end{equation}
before boundary conditions. The count is obtained from the rank of the sparse smoothness matrix and needs no determining set. The same computation for the first N\'ed\'elec space gives $\dim\Ncal_d(T)=(d+1)(d+3)$ and $(d+1)N_E+d(d+1)N_T$ after tangential assembly.

\section{Commuting identities and exactness}\label{sec:commuting}

\subsection{The edge identity}

\begin{lemma}[edge commuting identity]\label{lem:edge-commute}
Let $e$ be an edge of $T$ with unit tangent $t_e$, and let $u\in\Pp_{d+1}(T)$ have coefficient vector $c_T$. Let $r(u)\in\R^{d+2}$ be the coefficient vector of $u|_e$ and $r_t(\nabla u)\in\R^{d+1}$ the coefficient vector of $(\nabla u\cdot t_e)|_e$. Then
\begin{equation}\label{eq:localtracegrad}
   r_t(\nabla u)=|e|^{-1}\mathsf U_{d+1}\,r(u),
\end{equation}
that is, the coefficients of the tangential component of the gradient on $e$ are the Bernstein differences of the coefficients of $u$ on $e$.
\end{lemma}
\begin{proof}
Parameterize $e$ by $x(t)=v_1+t(v_2-v_1)$, $0\le t\le1$. By the chain rule $t_e\cdot\nabla u(x(t))=|e|^{-1}\frac{\dd}{\dd t}u(x(t))$, and $u(x(t))$ is the univariate polynomial with coefficients $r(u)$ by \eqref{eq:edge-restriction}. Formula \eqref{eq:univariate-derivative} gives the coefficients $|e|^{-1}\mathsf U_{d+1}r(u)$ of its derivative. The left side of \eqref{eq:localtracegrad} is the coefficient vector of the same univariate polynomial by the description after \eqref{eq:C1e}.
\end{proof}

\begin{theorem}[constraint-level commutation]\label{thm:constraint-commute}
For the global smoothness matrices,
\begin{equation}\label{eq:C1GBC0}
   C_1^{\curl}D_0^{\rm br}=B_0C_0,\qquad B_0=\diag_e\bigl(|e|^{-1}\mathsf U_{d+1}\bigr),
\end{equation}
and consequently $D_0^{\rm br}(\ker C_0)\subset\ker C_1^{\curl}$.
\end{theorem}
\begin{proof}
Apply Lemma~\ref{lem:edge-commute} on both triangles adjacent to an interior edge, with the edge coefficient vectors of both triangles ordered along the global tangent $t_e$; on a triangle whose local vertex order runs against $t_e$ this means composing its coefficient vector with the reversal permutation, which commutes with \eqref{eq:Delta} up to the sign that the reversed tangent carries. Subtract the two identities. The jump of the tangential coefficients of the gradient is the difference matrix applied to the jump of the scalar coefficients, which is $C_{1,e}^{\curl}D_0^{\rm br}=|e|^{-1}\mathsf U_{d+1}C_{0,e}$. Stacking over edges gives \eqref{eq:C1GBC0}, and $C_0c=0$ gives $C_1^{\curl}D_0^{\rm br}c=B_0C_0c=0$.
\end{proof}

The proof is local to one edge and one triangle, and it proceeds without a basis of $S^0_{d+1}(\tri)$. The identity \eqref{eq:localtracegrad} is also a sparse matrix equality that an implementation can check to roundoff on every triangle; Section~\ref{sec:numerics} does so for degrees one to five.

\subsection{The planar complexes and their exactness}

Let $D_1$ be the broken curl from $[S^{-1}_d(\tri)]^2$ to $S^{-1}_{d-1}(\tri)$; no smoothness matrix is needed at the last slot. Set $D_{0,h}=D_0^{\rm br}|_{V_h^0}$ and $D_{1,h}=D_1|_{V^1_{h,\curl}}$.

\begin{corollary}[subcomplex]\label{cor:subcomplex}
The sequence
\begin{equation}\label{eq:global-complex}
   \R\hookrightarrow V_h^0\xrightarrow{D_{0,h}}V_{h,\curl}^1\xrightarrow{D_{1,h}}V_h^2\to0
\end{equation}
is a subcomplex of $\R\hookrightarrow H^1(\Omega)\xrightarrow{\nabla}H(\curl;\Omega)\xrightarrow{\curl}L^2(\Omega)\to0$, and $D_{1,h}D_{0,h}=0$.
\end{corollary}
\begin{proof}
The inclusion $D_{0,h}V_h^0\subset V^1_{h,\curl}$ is Theorem~\ref{thm:constraint-commute}, and $\curl\nabla=0$ on each triangle gives $D_1D_0^{\rm br}=0$.
\end{proof}

\begin{theorem}[exactness on a simply connected domain]\label{thm:global-exact}
Let $\Omega$ be simply connected with boundary, let $\tri$ be a regular triangulation of $\Omega$, and let the spaces in \eqref{eq:global-complex} carry no essential boundary conditions. Then \eqref{eq:global-complex} is exact: $\ker D_{0,h}$ consists of the constants, $\ker D_{1,h}=\range D_{0,h}$, and $D_{1,h}V^1_{h,\curl}=V_h^2$. For a domain with holes the dimension of $\ker D_{1,h}/\range D_{0,h}$ equals the first Betti number of $\Omega$.
\end{theorem}
\begin{proof}
Let $v_h\in V^1_{h,\curl}$ with $\curl v_h=0$. By exactness of \eqref{eq:local-full} there is $u_T\in\Pp_{d+1}(T)$ with $\nabla u_T=v_h|_T$ on each triangle. Across an interior edge tangential continuity gives $D_{t_e}(u_{T^+}-u_{T^-})=0$ on $e$, so the jump of the potentials is a constant on $e$. Choose the additive constants along a spanning tree of the dual graph of $\tri$ so that the potentials agree on every tree edge. Around any cycle of the dual graph the accumulated jump equals the circulation of $v_h$ around the corresponding loop in $\Omega$, which vanishes by Stokes' theorem because $v_h$ is curl free and $\Omega$ is simply connected. The adjusted potentials agree on every interior edge and define $u_h\in V_h^0$ with $\nabla u_h=v_h$. For surjectivity, $\dim V_h^0=N_V+dN_E+\binom d2N_T$, the count of domain points of degree $d+1$ at vertices, in the interiors of edges and in the interiors of triangles \cite[Chap.~5]{LaiSchumaker2007}, $\dim V_h^2=\binom{d+1}2N_T$, and \eqref{eq:dimV1} together with Euler's relation $N_V-N_E+N_T=1$ give $\dim V^1_{h,\curl}=(\dim V_h^0-1)+\dim V_h^2$. Since $\ker D_{1,h}=\range D_{0,h}$ has dimension $\dim V_h^0-1$, rank--nullity gives $\rank D_{1,h}=\dim V_h^2$. For a domain with holes the cycle condition fails exactly on representatives of the first cohomology, and Theorem~\ref{thm:conformity} identifies the space with the classical one, whose discrete cohomology is that of the mesh.
\end{proof}

The same argument applies to the first N\'ed\'elec space: the identity \eqref{eq:localtracegrad} is a statement about polynomial traces and does not depend on the interior basis, the local sequence with $\Ncal_d$ is exact, and the count $(d+1)N_E+d(d+1)N_T=(\dim V_h^0-1)+\binom{d+2}2N_T$ follows from Euler's relation.

\subsection{Normal continuity by rotation}

With $n_e=Rt_e$ one has $n_e\cdot R\nabla u=t_e\cdot\nabla u$, so the difference matrix that controls tangential continuity of gradients also controls normal continuity of rotated gradients.

\begin{theorem}[rotated complex]\label{thm:hdiv-complex}
$C_1^{\divv}RD_0^{\rm br}=B_0C_0$ with the same $B_0$ as in \eqref{eq:C1GBC0}. Consequently
\[
   \R\hookrightarrow V_h^0\xrightarrow{R\nabla}V^1_{h,\divv}\xrightarrow{\divv}V_h^2\to0
\]
is a subcomplex, exact on simply connected domains, whose middle space is the BDM family; replacing $[\Pp_d]^2$ by $\RT_d$ gives the Raviart--Thomas family with the same normal-coefficient functionals.
\end{theorem}
\begin{proof}
The identity follows from $n_e\cdot R\nabla u=t_e\cdot\nabla u$ and Lemma~\ref{lem:edge-commute}; $\divv R\nabla=0$ gives the complex, and exactness follows from Theorem~\ref{thm:global-exact} by rotation.
\end{proof}

\section{Higher smoothness and smooth complexes}\label{sec:smooth}

The spaces of Section~\ref{sec:constraints} impose only the continuity required by the Sobolev space. Spline methods often require more, and the same edge storage holds $C^r$ smoothness functionals and de Rham functionals side by side. In this section we index scalar spaces by their own degree.

\subsection{A row-space criterion}

Let $\widetilde V_h^k$ be broken spaces, let $D_k:\widetilde V_h^k\to\widetilde V_h^{k+1}$ be an elementwise differential operator, and let $J_k$ be any collection of smoothness functionals, not necessarily independent, with $V_h^k=\ker J_k$.

\begin{theorem}[row-space criterion]\label{thm:rowspace}
The following are equivalent:
\begin{enumerate}[label=(\roman*)]
\item $D_k(\ker J_k)\subset\ker J_{k+1}$;
\item $\range\bigl((J_{k+1}D_k)^T\bigr)\subset\range(J_k^T)$;
\item there is a matrix $B_k$ with $J_{k+1}D_k=B_kJ_k$.
\end{enumerate}
If the smoothness functionals and the differential operator are assembled edge by edge and (iii) holds on every edge patch, then it holds globally with $B_k$ block diagonal by edges.
\end{theorem}
\begin{proof}
(i) says $\ker J_k\subset\ker(J_{k+1}D_k)$, and in finite dimensions $\ker A\subset\ker B$ if and only if $\range B^T\subset\range A^T$, which is (ii). Row-space inclusion means that every row of $J_{k+1}D_k$ is a combination of rows of $J_k$, which is (iii). The local statement holds because each row of $J_{k+1}D_k$ is supported on the triangles incident to one edge.
\end{proof}

\begin{corollary}[local diagnostic]\label{cor:localdiagnostic}
With $J_k^+$ the Moore--Penrose pseudoinverse, the pair $(J_k,J_{k+1})$ is compatible if and only if $J_{k+1}D_k(I-J_k^+J_k)=0$, and this can be tested on each edge patch separately.
\end{corollary}

The criterion is a design tool. For a candidate smooth vector spline space one assembles the smoothness functionals, differentiates the broken coefficients, and tests the corollary locally. If the test fails, the profile cannot form a subcomplex; if it passes, the subcomplex property is established before any dimension formula or determining set is known. Exactness is a separate question, and the criterion does not address it.

A profile that fails the test is $S^r_d(\tri)\to[S^r_{d-1}(\tri)]^2$ for $r\ge1$, the same smoothness on both sides: the gradient of a $C^r$ spline is only $C^{r-1}$, so the order-$r$ functionals of the gradient are not in the span of the order-$r$ functionals of the scalar, and the local residual $J_{k+1}D_k(I-J_k^+J_k)$ is nonzero on every edge patch. The dependence among the scalar functionals around a vertex is no obstacle to the test, since Theorem~\ref{thm:rowspace} is stated for row spaces and needs no independence.

\subsection{Explicit factors for componentwise \texorpdfstring{$C^r$}{Cr} profiles}

The componentwise profile
\begin{equation}\label{eq:strong-profile}
   S^r_d(\tri)\xrightarrow{\nabla}[S^{r-1}_{d-1}(\tri)]^2\xrightarrow{\curl}S^{r-2}_{d-2}(\tri),
\end{equation}
with the convention that $S^s$ is discontinuous for $s<0$, is the natural smooth analogue of \eqref{eq:global-complex}, and it is the sequence studied by Alfeld and Sorokina \cite{AlfeldSorokina2016}. It is a subcomplex because differentiation lowers the smoothness by one; the following theorem gives the factor $B_k$ of Theorem~\ref{thm:rowspace} explicitly, with Theorem~\ref{thm:constraint-commute} as the case $r=0$.

Fix an interior edge $e$ with unit tangent $t_e$ and unit normal $n_e=Rt_e$, so that $(t_e,n_e)$ is positively oriented. For a scalar spline $s$ let $\gamma_m(s)$, $m=0,\dots,r$, be the coefficient vectors of $(D_{n_e}^ms)|_e$ from the side of $T$, as at the end of Section~\ref{sec:bb}, and write $\Gamma^{(r)}_{e,T}$ for the map $c_T\mapsto(\gamma_0(s),\dots,\gamma_r(s))$. For a vector spline $v$ write $v=v_tt_e+v_nn_e$ with $v_t=v\cdot t_e$ and $v_n=v\cdot n_e$ on $T$, and let $\Gamma^{(r-1)}_{e,T}$ map its coefficients to $(\gamma_0(v_t),\gamma_0(v_n),\dots,\gamma_{r-1}(v_t),\gamma_{r-1}(v_n))$. Since $t_e$ and $n_e$ are constant, the frame components are fixed invertible combinations of the Cartesian components, so the kernel of the assembled functionals is $[S^{r-1}_{d-1}(\tri)]^2$ in either description. The jump functionals are $J^{(r)}_{0,e}=[\,\Gamma^{(r)}_{e,T^+}\ \ -\Gamma^{(r)}_{e,T^-}\,]$ and $J^{(r-1)}_{1,e}$ correspondingly; their row spaces coincide with those of the coefficient conditions \eqref{eq:LS-smoothness}, applied componentwise for the vector case.

\begin{theorem}[jet commutation]\label{thm:jet-commute}
Let $r\ge1$ and $d\ge r+1$, and let $G_T$ denote the gradient matrix from $\Pp_d(T)$ to $[\Pp_{d-1}(T)]^2$. On every edge $e$ of $T$,
\begin{equation}\label{eq:jet-local}
   \Gamma^{(r-1)}_{e,T}G_T=\mathsf B^{(r)}_e\,\Gamma^{(r)}_{e,T},
\end{equation}
where $\mathsf B^{(r)}_e$ has $2r$ block rows and $r+1$ block columns and its only nonzero blocks are
\[
   \bigl(\mathsf B^{(r)}_e\bigr)_{2m+1,\,m}=|e|^{-1}\mathsf U_{d-m},\qquad
   \bigl(\mathsf B^{(r)}_e\bigr)_{2m+2,\,m+1}=I,\qquad m=0,\dots,r-1 .
\]
Consequently $J^{(r-1)}_1D_0^{\rm br}=B_0^{(r)}J^{(r)}_0$ with $B_0^{(r)}=\diag_e\mathsf B^{(r)}_e$, and $\nabla S^r_d(\tri)\subset[S^{r-1}_{d-1}(\tri)]^2$. In the same way, for the broken curl $D_{1}$,
\begin{equation}\label{eq:jet-curl}
   \Gamma^{(r-2)}_{e,T}D_{1,T}=\mathsf B^{(r-1)}_{e,\curl}\,\Gamma^{(r-1)}_{e,T},\qquad
   \bigl(\mathsf B^{(r-1)}_{e,\curl}\bigr)_{m,\,2m+2}=|e|^{-1}\mathsf U_{d-1-m},\quad
   \bigl(\mathsf B^{(r-1)}_{e,\curl}\bigr)_{m,\,2m+3}=-I,
\end{equation}
for $m=0,\dots,r-2$, and \eqref{eq:strong-profile} is a subcomplex with edge-local factors.
\end{theorem}
\begin{proof}
Let $s\in\Pp_d(T)$ and $v=\nabla s$. Since $t_e$ and $n_e$ are constant, $D_{t_e}$ and $D_{n_e}$ commute, and $(D_{t_e}w)|_e=|e|^{-1}\frac{\dd}{\dd t}(w|_e)$ for every polynomial $w$ by the chain rule. Hence for $0\le m\le r-1$,
\[
   (D_{n_e}^mv_t)|_e=(D_{n_e}^mD_{t_e}s)|_e=|e|^{-1}\tfrac{\dd}{\dd t}\bigl((D_{n_e}^ms)|_e\bigr),
   \qquad
   (D_{n_e}^mv_n)|_e=(D_{n_e}^{m+1}s)|_e .
\]
The polynomial $(D_{n_e}^ms)|_e$ has degree $d-m$, so by \eqref{eq:univariate-derivative} the first identity reads $\gamma_m(v_t)=|e|^{-1}\mathsf U_{d-m}\gamma_m(s)$, and the second reads $\gamma_m(v_n)=\gamma_{m+1}(s)$. These $2r$ identities are the block rows of \eqref{eq:jet-local}. Applying \eqref{eq:jet-local} on $T^+$ and $T^-$ with the same frame and subtracting gives $J^{(r-1)}_{1,e}D_0^{\rm br}=\mathsf B^{(r)}_eJ^{(r)}_{0,e}$, and stacking over edges gives the global identity; the inclusion follows as in Theorem~\ref{thm:constraint-commute}. For the curl, $\curl v=D_{t_e}v_n-D_{n_e}v_t$ on $T$ because the frame is orthonormal and positively oriented, so for $0\le m\le r-2$,
\[
   \gamma_m(\curl v)=|e|^{-1}\mathsf U_{d-1-m}\gamma_m(v_n)-\gamma_{m+1}(v_t),
\]
which is \eqref{eq:jet-curl}. The subcomplex statement follows by stacking, and $D_1D_0^{\rm br}=0$ holds on each triangle.
\end{proof}

The factor is as sparse as in the $r=0$ case: block $m$ of the vector functionals depends only on blocks $m$ and $m+1$ of the scalar functionals. If the smoothness functionals are stored in the coefficient form \eqref{eq:LS-smoothness} rather than the derivative form, the factor changes by the row-basis matrices of Proposition~\ref{prop:row-compression} below and remains edge local. If a transverse direction other than $n_e=Rt_e$ is used, only the sign of the identity blocks in \eqref{eq:jet-curl} changes.

The theorem establishes the subcomplex property of \eqref{eq:strong-profile} for every $r\ge1$ and $d\ge r+1$, and Theorem~\ref{thm:constraint-commute} covers $r=0$. Exactness is a property of the triangulation, and the row-space test gives no information about it. For $r\ge1$ the dimension of $S^r_d(\tri)$ depends on the geometry when $d<3r+2$, through singular vertices and near-singular configurations \cite[Chap.~9]{LaiSchumaker2007}, so the ranks of the smoothness matrices, and with them the cohomology of \eqref{eq:strong-profile}, can change under a perturbation of the vertices. Exact smooth sequences may then require macro-refinements, supersmoothness at vertices, or a modified terminal space, as the Powell--Sabin sequence below shows \cite{GuzmanLischkeNeilan2020,GuzmanLischkeNeilan2022}. The profile $S^1_5\to[S^0_4]^2\to S^{-1}_3$ used in the numerical comparison of Section~\ref{sec:three-realizations-numerical} is a subcomplex by the theorem and is used there to test the representation of a $C^1$ smoothness matrix; its exactness is a separate question that the comparison does not need.

\subsection{The Powell--Sabin exact complex in broken coefficients}\label{sec:ps-complex}

For an exact smooth example we use the lowest-order Powell--Sabin sequence of Guzm\'an, Lischke and Neilan \cite{GuzmanLischkeNeilan2020}. Let $\tri_{\rm PS}$ be the Powell--Sabin refinement of a triangulation $\tri$ of a simply connected domain, obtained by splitting each triangle into six subtriangles. Their global sequence is
\begin{equation}\label{eq:ps-complex}
   \R\hookrightarrow S^1_2(\tri_{\rm PS})\xrightarrow{\rot}L^1_1(\tri_{\rm PS})\xrightarrow{\divv}V^2_0(\tri_{\rm PS})\to0,
\end{equation}
where $S^1_2(\tri_{\rm PS})$ is the classical $C^1$ quadratic Powell--Sabin space, $L^1_1$ is the space of continuous piecewise linear vector fields on $\tri_{\rm PS}$, and $V^2_0$ is the space of piecewise constants satisfying the singular-vertex condition of \cite[Sec.~5]{GuzmanLischkeNeilan2020}. Exactness on simply connected domains is \cite[Thm.~5.4]{GuzmanLischkeNeilan2020}.

\paragraph{The refinement.}
For each triangle $T$ of $\tri$ we take the split point $z_T$ at the incenter. For two triangles sharing an edge, the segment joining their incenters crosses that edge at a point $z_e$; on a boundary edge $z_e$ is the midpoint. Joining $z_T$ to the three vertices of $T$ and to the three points $z_e$ gives the six subtriangles. Incenters are a standard admissible choice \cite[Sec.~1]{GuzmanLischkeNeilan2020}. The points $z_e$ on interior edges are singular vertices, since the four subedges meeting there lie on two lines, and this is what the terminal space uses.

\paragraph{The smoothness functionals.}
Let $e$ be an interior edge of $\tri_{\rm PS}$ with a fixed transverse unit vector $n_e$. For a quadratic $q$ with coefficients $c$ on a subtriangle, the coefficient vector of $q|_e$ is the vector of coefficients at the three domain points on $e$, and the coefficient vector of $(D_{n_e}q)|_e$ is, by \eqref{eq:bern-deriv},
\begin{equation}\label{eq:ps-normal-derivative}
   \bigl(\gamma_1(q)\bigr)_\beta=2\sum_{i=1}^3(D_{n_e}\lambda_i)\,c_{\beta+e_i},
\end{equation}
for the two multi-indices $\beta$ of degree one on $e$. The scalar block $J_{0,e}$ is the jump of these five coefficients across $e$; equality of the quadratic trace already gives equality of its tangential derivative, so equality of one transverse derivative gives equality of the gradient, and $\ker J_0=S^1_2(\tri_{\rm PS})$. This is the derivative form of \eqref{eq:LS-smoothness} with $r=1$. For a piecewise linear vector field the block $J_{1,e}$ is the componentwise jump of the two coefficients on $e$, and $\ker J_1=L^1_1(\tri_{\rm PS})$. For an interior singular point $z$ with the four incident subtriangles $T_1,\dots,T_4$ in cyclic order, the functional
\begin{equation}\label{eq:ps-J2}
   \theta_z(q)=q|_{T_1}-q|_{T_2}+q|_{T_3}-q|_{T_4}
\end{equation}
on piecewise constants gives, stacked over singular points, the matrix $J_2$ with $\ker J_2=V^2_0(\tri_{\rm PS})$ \cite[Sec.~5, Rem.~5.1]{GuzmanLischkeNeilan2020}. The three kernels are therefore the three spaces of \eqref{eq:ps-complex}, and not merely spaces of the same dimensions.

With $\mathsf R_{\rm PS}$ and $\mathsf D_{\rm PS}$ the broken rotated gradient and divergence matrices, the two compatibility relations are tested by
\begin{equation}\label{eq:ps-tests}
   J_1\mathsf R_{\rm PS}(I-J_0^+J_0)=0,\qquad J_2\mathsf D_{\rm PS}(I-J_1^+J_1)=0,
\end{equation}
and exactness by comparing ranks. Section~\ref{sec:ps-numerics} reports these checks; they confirm that the implementation reproduces the known exact sequence.

\begin{remark}[partially smooth vector spaces]
The two frame components of a vector spline may be given different orders of transverse smoothness across an edge. Such profiles are tested by Theorem~\ref{thm:rowspace}. They are related in spirit to partially discontinuous nodal elements \cite{HuHuZhang2022}, expressed here as smoothness functionals on B-coefficients, so that they combine with any scalar spline conditions in the same storage.
\end{remark}

\section{The constrained Galerkin problem and its realizations}\label{sec:variational}

\subsection{The constrained equations}

Let $\widetilde V_h\cong\R^N$ be a broken coefficient space, $A\in\R^{N\times N}$ the assembled broken Galerkin matrix, $f\in\R^N$ the load vector, and $Cc=g$, $C\in\R^{m\times N}$, the collection of smoothness and boundary equations. The rows of the smoothness matrix $C$ are allowed to be dependent, and we assume only that the affine set $\mathcal K_g=\{c:Cc=g\}$ is nonempty. The constrained Galerkin problem is
\begin{equation}\label{eq:conforming-abstract}
   c\in\mathcal K_g,\qquad v^T(Ac-f)=0\quad\forall v\in\ker C,
\end{equation}
which for $g=0$ is Galerkin projection onto $\ker C$. In matrix form,
\begin{equation}\label{eq:KKT}
   \begin{bmatrix}A&C^T\\C&0\end{bmatrix}\begin{bmatrix}c\\\lambda\end{bmatrix}=\begin{bmatrix}f\\g\end{bmatrix}.
\end{equation}
This is the constrained B-form system of \cite[Sec.~3]{AwanouLaiWenston2006}; saddle-point linear algebra is reviewed in \cite{BenziGolubLiesen2005}. The collocation method of \cite{LaiLee2022} replaces the Galerkin equations $Ac=f$ by the strong form evaluated at collocation points, keeping the same constraint rows $Cc=g$, and solves the constrained least-squares problem; the realizations of Section~\ref{sec:three-realizations} apply to that system as well.

\begin{theorem}[equivalence with the conforming method]\label{thm:kkt-equivalence}
Let the broken matrix $A$ be symmetric and $\mathcal K_g$ nonempty. Then $c$ solves \eqref{eq:conforming-abstract} if and only if there is a multiplier $\lambda$ with $(c,\lambda)$ solving \eqref{eq:KKT}. If $v^TAv>0$ for every nonzero $v\in\ker C$, then $c$ is unique. If the rows of $C$ are dependent, the multiplier is not unique and the block matrix in \eqref{eq:KKT} is singular in the corresponding multiplier directions.
\end{theorem}
\begin{proof}
If \eqref{eq:KKT} holds then $c\in\mathcal K_g$ and $v^T(Ac-f)=-(Cv)^T\lambda=0$ for $v\in\ker C$. Conversely, if \eqref{eq:conforming-abstract} holds then $Ac-f\perp\ker C=(\range C^T)^\perp$, so $Ac-f=-C^T\lambda$ for some $\lambda$. Two primal solutions differ by $w\in\ker C$ with $w^TAw=0$, so $w=0$ under the positivity condition. If $0\ne\eta\in\ker C^T$ then $(0,\eta)$ is a null vector of the block matrix.
\end{proof}

Two sources of singularity occur in spline computations and should be kept apart. The broken matrix $A$ is often singular because the elementwise energy has polynomial null modes: constants for the Poisson stiffness, gradients for the curl--curl matrix, linears for the Hessian energy. Separately, dependent rows of the smoothness matrix make $C^T$ noninjective. Awanou, Lai and Wenston allow both, and their convergence theorem assumes that $A$ is nonnegative and positive definite with respect to the constraints, that is $x^TAx=0$ and $Cx=0$ imply $x=0$ \cite[Thm.~6]{AwanouLaiWenston2006}. For the mathematical statements we keep the full smoothness matrix. For a direct factorization one may replace $C$ by a matrix $\bar C$ with full row rank and the same row space; then $\ker\bar C=\ker C$, the primal solution is unchanged, and only redundant multiplier coordinates are removed. This row compression is not a null-space reduction and eliminates no B-coefficient.

\begin{proposition}[row-basis invariance]\label{prop:row-compression}
Let $\bar C_k$ have the same row space as $C_k$ in every slot. Then $\ker\bar C_k=\ker C_k$, so the spaces and their cohomology are unchanged, and if $C_{k+1}D_k=B_kC_k$ then $\bar C_{k+1}D_k=\bar B_k\bar C_k$ for some matrix $\bar B_k$, which is in general less local because the row operations that form $\bar C_k$ mix functionals of different edges.
\end{proposition}
\begin{proof}
There are matrices $R_k,L_k$ with $\bar C_k=R_kC_k$ and $C_k=L_k\bar C_k$, so $\bar C_{k+1}D_k=R_{k+1}B_kL_k\bar C_k$.
\end{proof}

This is why compatibility is proved on the raw edge functionals, where the factor is block diagonal, and rows are compressed only as a solver preprocessing step. Two rank-revealing procedures are used in the experiments: a sparse elimination in assembly order with pivot tolerance $2\times10^{-11}$ for the $C^1$, $d=5$ comparison, and QR with column pivoting on $C^T$ for the Powell--Sabin biharmonic problems. Discarded rows are always checked through the full residual $Cc$.

\subsection{Three realizations of the same equations}\label{sec:three-realizations}

\paragraph{The null-space method.}
Choose a particular $c_p$ with $Cc_p=g$ and a matrix $Z$ of full column rank with $\range Z=\ker C$. Every admissible coefficient vector is $c=c_p+Z\hat c$, and \eqref{eq:conforming-abstract} becomes
\begin{equation}\label{eq:Zreduced}
   Z^TAZ\,\hat c=Z^T(f-Ac_p).
\end{equation}
This is the null-space method for \eqref{eq:KKT} \cite[Sec.~6]{BenziGolubLiesen2005}. In the language of \cite[Chap.~5]{LaiSchumaker2007}, the columns of the null-space matrix $Z$ extend the coefficients on a determining set to all B-coefficients, and a stable local minimal determining set corresponds to a sparse $Z$ with local support. In the experiments $Z$ is built by locality-oriented sparse rank-revealing elimination of the smoothness equations, not by a dense factorization. The reduced matrix $Z^TAZ$ is symmetric positive definite when $A$ is positive definite on $\ker C$, and its size is the dimension of the spline space.

\begin{proposition}[the reduced complex]\label{prop:z-complex}
Let $Z_k$ have $\range Z_k=\ker C_k$ in every slot of a compatible constraint complex. Then there are matrices $\widehat D_k$ with
\begin{equation}\label{eq:z-commute}
   D_kZ_k=Z_{k+1}\widehat D_k,
\end{equation}
and $\widehat D_{k+1}\widehat D_k=0$ whenever $D_{k+1}D_k=0$. The reduced sequence is therefore the same discrete complex in the coordinates of the determining sets, and its cohomology is that of the constraint complex.
\end{proposition}
\begin{proof}
Compatibility gives $D_kZ_k\subset\ker C_{k+1}=\range Z_{k+1}$, and $Z_{k+1}$ has full column rank, so $\widehat D_k$ exists and is unique. Then $Z_{k+2}\widehat D_{k+1}\widehat D_k=D_{k+1}D_kZ_k=0$ gives $\widehat D_{k+1}\widehat D_k=0$, and the isomorphisms $\hat c\mapsto Z_k\hat c$ identify kernels and ranges slotwise.
\end{proof}

\paragraph{Augmented Lagrangian iteration.}
Set $E_\eps=A+\eps^{-1}C^TC$. Awanou, Lai and Wenston eliminate the multiplier from their augmented saddle iteration and obtain, from an initial multiplier $\lambda^{(0)}$,
\begin{equation}\label{eq:alw}
   c^{(1)}=E_\eps^{-1}\bigl(f+\eps^{-1}C^Tg-C^T\lambda^{(0)}\bigr),\qquad
   c^{(m+1)}=E_\eps^{-1}\bigl(Ac^{(m)}+\eps^{-1}C^Tg\bigr),\quad m\ge1,
\end{equation}
which are the eliminated forms of their equations (12)--(13) \cite[Algorithm~5]{AwanouLaiWenston2006}. Under the relative positivity condition above, $E_\eps$ is invertible for every $\eps>0$ and the iteration converges linearly \cite[Thm.~6]{AwanouLaiWenston2006}, with a contraction factor that decreases with $\eps$ and is analyzed in \cite{AwanouLai2005}; the price of a small $\eps$ is the conditioning of $E_\eps$, which the experiments below measure. At a finite iterate $Cc^{(m)}\ne g$ in general, and the identity $C_{k+1}D_kc^{(m)}=B_kC_kc^{(m)}$ shows how the constraint residual is transported into the differentiated slot. Since $E_\eps$ contains $C^TC$, duplicating or rescaling a row of $C$ leaves the constraint set unchanged but changes the iteration and its conditioning.

\paragraph{Direct solution of the saddle system.}
The third realization keeps the broken coefficients and solves \eqref{eq:KKT} directly, after compressing redundant multiplier rows for the factorization. It constructs no null-space matrix and introduces no parameter, at the price of an indefinite bordered system. Its advantage in a complex is that no $Z_k$ has to be built and kept compatible in every slot, which is what the eigenvalue and surface experiments below use.

\subsection{An interface system for semidefinite element matrices}

Suppose the broken matrix $A=\diag_TA_T$ is symmetric positive semidefinite and let the columns of $R_T$ span $\ker A_T$. Set $A^+=\diag_TA_T^+$ with the Moore--Penrose pseudoinverse, $R=\diag_TR_T$, $H=CA^+C^T$ and $G=CR$.

\begin{proposition}[interface system]\label{prop:pseudoinverse}
After lifting nonzero data, the homogeneous problem $Ac+C^T\lambda=f$, $Cc=0$ is equivalent to
\begin{equation}\label{eq:HG}
   \begin{bmatrix}H&G\\G^T&0\end{bmatrix}\begin{bmatrix}\lambda\\\beta\end{bmatrix}
   =\begin{bmatrix}CA^+f\\R^Tf\end{bmatrix},\qquad c=A^+(f-C^T\lambda)-R\beta .
\end{equation}
Redundant rows of $C$ give null directions in the multiplier block, and a row basis may be used for factorization without changing $c$.
\end{proposition}
\begin{proof}
The elementwise equation $Ac=f-C^T\lambda$ is solvable if and only if $R^T(f-C^T\lambda)=0$, that is $G^T\lambda=R^Tf$, and its general solution is $c=A^+(f-C^T\lambda)-R\beta$. Substituting into $Cc=0$ gives the first row of \eqref{eq:HG}.
\end{proof}

The interface matrix $H$ is assembled from the element contributions $C_TA_T^+C_T^T$, and the coarse block $G$ carries the element null modes. When every $A_T$ is invertible the system reduces to $CA^{-1}C^T\lambda=CA^{-1}f$, which is positive definite for a row basis of $C$ if $A$ is positive definite. For shifted Maxwell problems the element matrices are positive definite; for the time-harmonic form they can be indefinite and the interface system must be treated as such.

\subsection{A constrained Maxwell pencil}

Let $K$ and $M$ be the broken curl--curl and mass matrices and let $C$ impose tangential continuity and the perfect conductor condition. Throughout this subsection $C$ is a row basis of the raw smoothness matrix, so it has full row rank; before compression the dependent rows make the bordered matrix singular in the multiplier directions, and the pencil below is regular only after compression. The conforming eigenproblem is
\begin{equation}\label{eq:maxeig-conf}
   u_h\in\ker C\setminus\{0\},\qquad v^TKu_h=\lambda_hv^TMu_h\quad\forall v\in\ker C .
\end{equation}

\begin{theorem}[constrained pencil]\label{thm:eigen-pencil}
Let $\mathcal A=\begin{bmatrix}K&C^T\\C&0\end{bmatrix}$ and $\mathcal M=\begin{bmatrix}M&0\\0&0\end{bmatrix}$. Every finite eigenvalue of $\mathcal Ax=\lambda\mathcal Mx$ with nonzero primal component is an eigenvalue of \eqref{eq:maxeig-conf}, and every eigenvalue of \eqref{eq:maxeig-conf} is a finite eigenvalue of the pencil. If the broken mass matrix $M$ is positive definite, then $\ker\mathcal M=\{0\}\times\R^m$; since $C$ has full row rank the pencil is regular, and the multiplier directions contribute only infinite eigenvalues.
\end{theorem}
\begin{proof}
For $x=(u,\mu)$ the pencil reads $Ku+C^T\mu=\lambda Mu$, $Cu=0$; multiplying by $v^T$ with $v\in\ker C$ removes the multiplier and gives \eqref{eq:maxeig-conf}. Conversely, if $u$ solves \eqref{eq:maxeig-conf} then $Ku-\lambda Mu\in(\ker C)^\perp=\range C^T$, which gives $\mu$. The mass block vanishes on the multiplier variables, and for a regular pencil the vectors in $\ker\mathcal M$ belong to the infinite generalized part, as the Weierstrass canonical form of the pair $(\mathcal A,\mathcal M)$ shows.
\end{proof}

For a shift $\sigma$ that is not a finite eigenvalue, the shift-invert operator $\mathcal T_\sigma=(\mathcal A-\sigma\mathcal M)^{-1}\mathcal M$ maps a finite eigenpair $(\lambda,x)$ to the eigenvalue $(\lambda-\sigma)^{-1}$ and annihilates $\ker\mathcal M$. A sparse factorization of the bordered matrix $\mathcal A-\sigma\mathcal M$ therefore gives a shift-invert eigensolver for the constrained spectrum without a conforming basis; details are in Appendix~\ref{app:eigs}.

\section{Curved triangulated surfaces}\label{sec:surface}

\subsection{Surface maps and Piola transformations}

Let $\Gamma\subset\R^3$ be a smooth oriented surface and let $T\subset\Gamma$ be a curved triangle parameterized by a regular map $F_T:\wideT\to T$ from the reference triangle. Set $J=DF_T\in\R^{3\times2}$, $G=J^TJ$ and $J_\Gamma=\sqrt{\det G}$. The maps are assumed uniformly regular and compatible on shared edges, so the curved triangles form a conforming triangulation of a surface $\Gamma_h$; when $\Gamma_h=\Gamma$ the geometry is exact, and otherwise comparison with $\Gamma$ carries the geometric error analyzed in \cite{HolstStern2012,Licht2023}. Scalars pull back by composition, $\hat u=u\circ F_T$. Tangent vector fields transform by the covariant Piola map
\begin{equation}\label{eq:covariant}
   v\circ F_T=JG^{-1}\hat v,\qquad\text{equivalently}\qquad J^T(v\circ F_T)=\hat v,
\end{equation}
flux fields by the contravariant Piola map $w\circ F_T=J_\Gamma^{-1}J\hat w$, and densities by $q\circ F_T=J_\Gamma^{-1}\hat q$.

\subsection{Invariance of the smoothness functionals}

\begin{theorem}[tangential invariance]\label{thm:surface-trace}
Let $\hat e\subset\partial\wideT$ be parameterized by $\hat\gamma(t)$ and $e=F_T(\hat e)$ by $\gamma=F_T\circ\hat\gamma$. If $v$ and $\hat v$ are related by \eqref{eq:covariant}, then $(v\circ\gamma)\cdot\gamma'(t)=(\hat v\circ\hat\gamma)\cdot\hat\gamma'(t)$. Hence, when the two element maps sharing $e$ agree on $\hat e$, equality of the B-coefficients of the tangential components on the two reference triangles is equivalent to equality of the tangential trace on the curved edge, and the reference tangential functionals are inherited on the surface without change.
\end{theorem}
\begin{proof}
$\gamma'=J\hat\gamma'$ and $J^Tv=\hat v$ give $v\cdot\gamma'=(J^Tv)\cdot\hat\gamma'=\hat v\cdot\hat\gamma'$.
\end{proof}

\begin{theorem}[flux invariance]\label{thm:surface-flux}
If $w$ and $\hat w$ are related by the contravariant map, then the co-normal flux density of $w$ times the physical line measure equals the normal flux density of $\hat w$ times the reference line measure. Hence the reference normal functionals enforce $H(\divv_\Gamma)$ conformity on the surface.
\end{theorem}
\begin{proof}
Let $N=(J_1\times J_2)/J_\Gamma$ be the unit normal, $\hat\tau=\hat\gamma'$ and $\hat\nu=R\hat\tau$. The physical co-normal line vector is $N\times J\hat\tau\dd\hat t$, and with $w=J\hat w/J_\Gamma$,
\[
   (w\cdot\mu)\dd s=J_\Gamma^{-1}N\cdot(J\hat\tau\times J\hat w)\dd\hat t=\det(\hat\tau,\hat w)\dd\hat t=(\hat w\cdot\hat\nu)\dd\hat t,
\]
using $N\cdot(J\hat a\times J\hat b)=J_\Gamma\det(\hat a,\hat b)$ for $\hat a,\hat b\in\R^2$.
\end{proof}

The smoothness matrices $C_0$, $C_1^{\curl}$ and $C_1^{\divv}$ are therefore assembled from reference data and reused on every geometry; the geometry enters the element mass and stiffness matrices only.

\begin{theorem}[surface commutation]\label{thm:surface-commute}
For smooth reference fields,
\[
   \nabla_\Gamma u\circ F_T=JG^{-1}\widehat\nabla\hat u,\qquad
   \curl_\Gamma v\circ F_T=J_\Gamma^{-1}\widehat\curl\hat v,\qquad
   \divv_\Gamma w\circ F_T=J_\Gamma^{-1}\widehat\divv\hat w,
\]
and the mapped constraint spaces form a de Rham subcomplex on the conforming curved triangulation.
\end{theorem}
\begin{proof}
Differentiating $\hat u=u\circ F_T$ in a reference direction $\hat\xi$ gives $\hat\xi\cdot\widehat\nabla\hat u=(J\hat\xi)\cdot\nabla_\Gamma u$, so $J^T\nabla_\Gamma u=\widehat\nabla\hat u$, and since $\nabla_\Gamma u$ is tangent the metric system gives the first identity. For the curl, Stokes' theorem on an arbitrary Lipschitz subregion $\hat\omega$ of $\wideT$ and the tangential invariance give $\int_{\hat\omega}\widehat\curl\hat v=\int_{\partial F_T(\hat\omega)}v\cdot\dd x=\int_{\hat\omega}(\curl_\Gamma v\circ F_T)J_\Gamma$, and $\hat\omega$ is arbitrary. The divergence identity follows in the same way from the flux invariance. The trace invariances then give the mapped subcomplex.
\end{proof}

If the element maps form an exact conforming parameterization of a surface, the mapped complex is isomorphic to the reference complex and has the same cohomology: exact on a contractible patch, and on a closed surface the cohomology dictated by its topology. These statements concern the de Rham functionals only. Higher-order smoothness across a curved edge depends on derivatives of the geometry map, the transfer of $C^r$ functionals to curved patches is open, and the $C^1$ Powell--Sabin experiments in this paper are planar. This is a limitation of the present work rather than of the representation.

\subsection{The curved geometries used below}

For the unit sphere, let $a_1,a_2,a_3\in S^2$ be the vertices of a triangle of an icosahedral triangulation. The radial map
\begin{equation}\label{eq:sphere-map}
   F_T(\lambda)=\frac{\lambda_1a_1+\lambda_2a_2+\lambda_3a_3}{\abs{\lambda_1a_1+\lambda_2a_2+\lambda_3a_3}}
\end{equation}
tiles the sphere exactly. For the hyperboloid we use the patch $\Phi(x,y)=(x,y,\sqrt{1+x^2+y^2})$, $(x,y)\in(0,1)^2$, composed with affine maps in the parameter plane; this is a bounded patch with boundary of the two-sheeted hyperboloid and tests curved metric assembly and trace conformity. Figure~\ref{fig:curved-geometries} shows both meshes.

\begin{figure}[htbp]
\centering
\includegraphics[width=.45\textwidth]{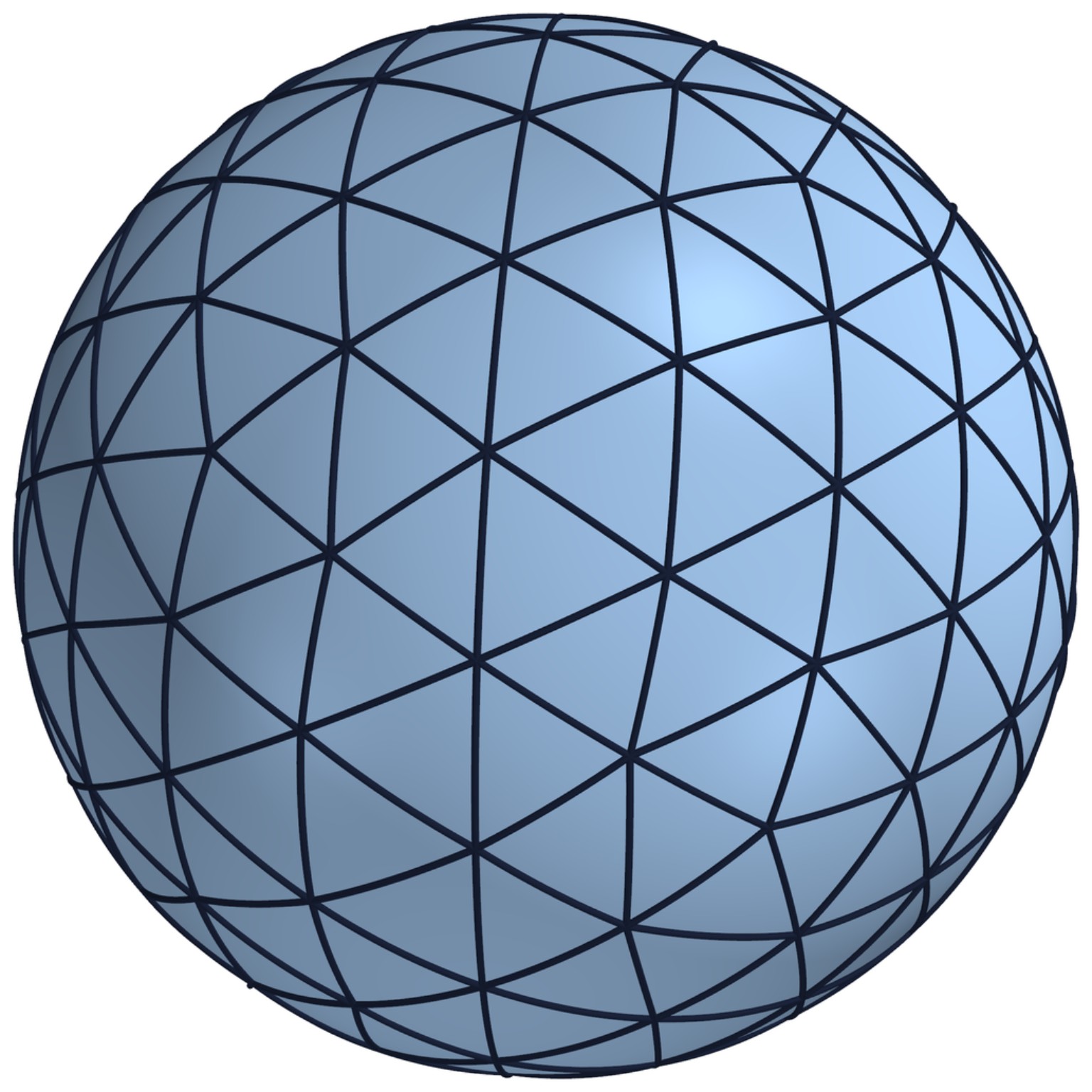}\hfill
\includegraphics[width=.45\textwidth]{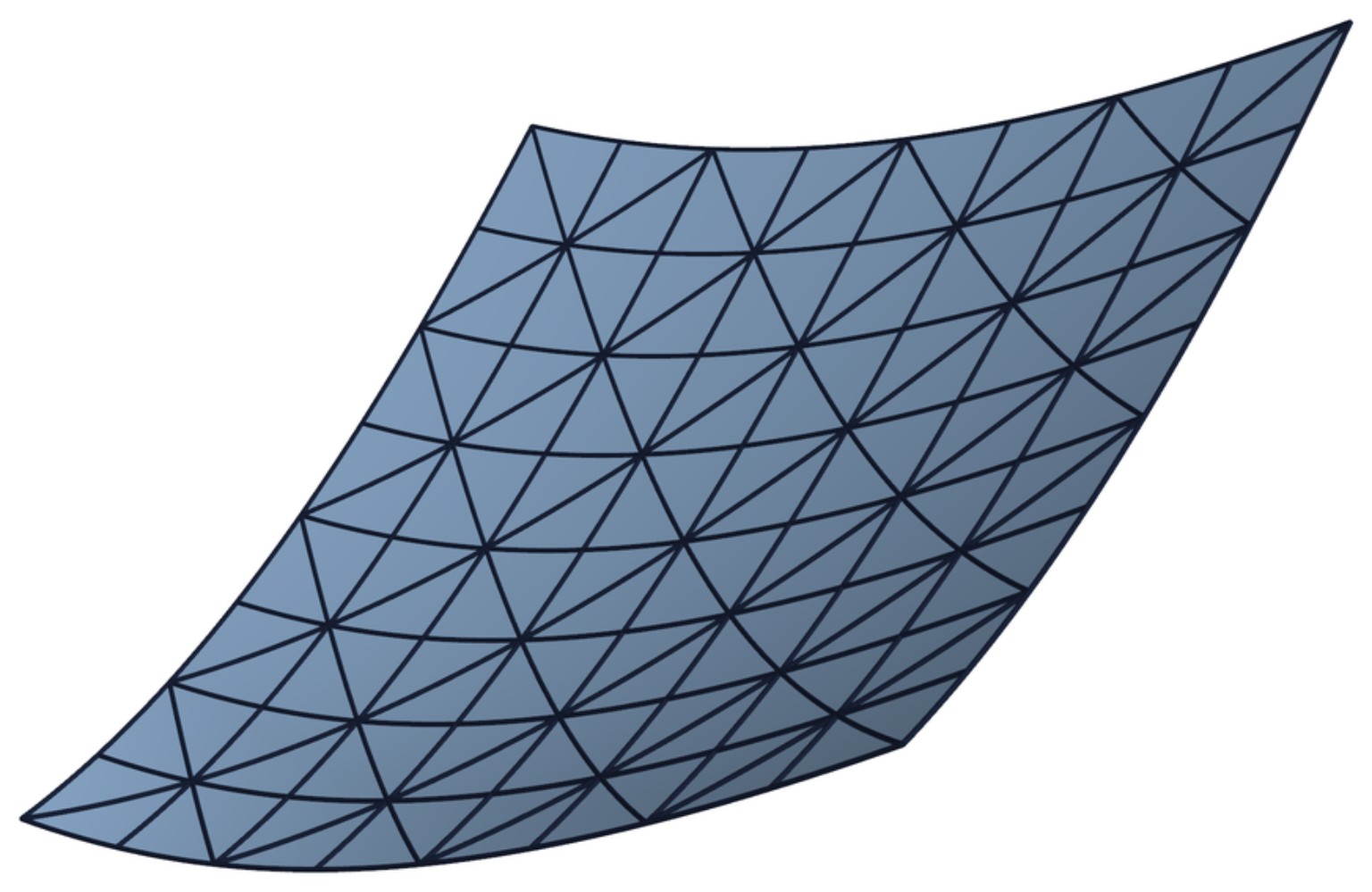}
\caption{Curved triangulations used in the experiments: the sphere under the radial map \eqref{eq:sphere-map} and the hyperboloid patch. The edge smoothness matrices are inherited unchanged from the reference triangulation.}
\label{fig:curved-geometries}
\end{figure}

\section{Approximation and spectral consequences}\label{sec:analysis}

The representation does not change the finite-dimensional space, so approximation results are inherited from the classical theory; we record what the experiments use.

Let $V^1_{h,d}$ be the tangentially continuous space of degree $d$ on a shape-regular triangulation, which by Theorem~\ref{thm:conformity} is the second N\'ed\'elec space. For $u\in H^s(\Omega)^2$ with $\curl u\in H^t(\Omega)$, $1\le s\le d+1$, $0\le t\le d$, the standard commuting projection gives $v_h\in V^1_{h,d}$ with
\begin{equation}\label{eq:approx}
   \norm{u-v_h}_{L^2}\le Ch^s\norm{u}_{H^s},\qquad\norm{\curl(u-v_h)}_{L^2}\le Ch^t\norm{\curl u}_{H^t},
\end{equation}
so smooth fields are approximated to order $d+1$ in $L^2$ and $d$ in curl \cite{ArnoldFalkWinther2006}. The rotated statement gives the BDM estimate with the divergence.

For the shifted problem $(\curl u,\curl v)+\alpha(u,v)=\ell(v)$, $\alpha>0$, the form is coercive on $H(\curl)$, the primal component of \eqref{eq:KKT} is the conforming Galerkin solution by Theorem~\ref{thm:kkt-equivalence}, and C\'ea's lemma with \eqref{eq:approx} gives curl error $O(h^d)$, with $L^2$ error $O(h^{d+1})$ under the usual duality argument. For the time-harmonic form $(\curl u,\curl v)-k^2(u,v)$ with $k^2$ not an eigenvalue, discrete compactness of the compatible family gives uniform stability for small $h$, and the same holds for the constrained representation since the space is the same; the element and interface matrices are then indefinite.

For the eigenproblem $(\curl u,\curl v)=\lambda(u,v)$ on $H_0(\curl)$ restricted to the complement of gradients, the nonzero eigenvalues on the unit square are $\pi^2(m^2+n^2)$, $(m,n)\ne(0,0)$, so the first cluster is $\pi^2,\pi^2,2\pi^2$. For an isolated eigenvalue $\lambda$ with eigenspace $E$ and $\delta_h(E)$ the $H(\curl)$ approximation defect of $E$, spectral approximation theory for the compatible family gives $\dist(E,E_h)\le C\delta_h(E)$ and $\abs{\lambda-\lambda_{h,j}}\le C\delta_h(E)^2$, hence eigenvalue error $O(h^{2d})$ for smooth eigenspaces, and the pencil of Theorem~\ref{thm:eigen-pencil} has exactly these finite eigenvalues. On the unit sphere the coexact fields $n\times\nabla_\Gamma Y_{\ell m}$ have eigenvalues $\ell(\ell+1)$ with multiplicity $2\ell+1$, so the first positive eigenvalue is $2$ with multiplicity three, and with the exact radial map the observed error is discretization error only.

\section{Numerical experiments}\label{sec:numerics}

\subsection{Setting}\label{sec:implementation}

All experiments store B-coefficients by triangle and smoothness functionals by edge. No global basis is formed except in the null-space method of the three-realization comparison. Source problems are solved from \eqref{eq:KKT} with a sparse direct factorization, eigenvalues from the shift-invert operator of Theorem~\ref{thm:eigen-pencil}. The raw smoothness matrix is used for every reported residual and commuting check; only multiplier rows are compressed when the bordered matrix is singular. Element integrals use a Duffy transform of a tensor Gauss--Legendre rule with $\max(6,d+4)$ points per direction for assembly and $\max(8,d+5)$ for error evaluation, except on the hyperboloid, where $d=1$ and a $7\times7$ rule is used throughout. Polynomial terms are therefore integrated exactly and trigonometric loads and curved metrics numerically. Every convergence line contains at least six independently computed resolutions. The implementation is in NumPy/SciPy and its scripts are listed in Appendix~\ref{app:repro}.

As a check on the curl--curl kernel, on the $8\times8$ mesh at $d=3$ the conforming scalar space with homogeneous boundary values has dimension $961$, the rank of the discrete gradient is $961$, the kernel of the conforming discrete curl has dimension $961$, and $\norm{D_1D_0}/(\norm{D_1}\norm{D_0})=4.8\times10^{-18}$; no additional near-zero modes were detected.

\subsection{Planar tangentially continuous source problem}\label{sec:flat-hcurl}

On $\Omega=(0,1)^2$ we solve $\curl\curl E-E=f$ with $E\cdot t=0$ on $\partial\Omega$ and the manufactured field
\begin{equation}\label{eq:manufactured-E}
   E(x,y)=\bigl(\sin(\pi y)\sin(2\pi x),\ \sin(\pi x)\sin(2\pi y)\bigr),
\end{equation}
on $n\times n$ Cartesian partitions split into triangles by alternating diagonals, with the vector field in $[\Pp_d]^2$ on every triangle. Table~\ref{tab:flat-source} and Figure~\ref{fig:flat-source} show the errors. The rates approach $d+1$ in $L^2$ and $d$ in curl, and the constraint residuals $\norm{Cu}/\norm u$ lie between $10^{-18}$ and $3.8\times10^{-17}$ in all eighteen runs. Separately, the local identity \eqref{eq:localtracegrad} assembled on a generic triangle has maximum residual entry between $6.3\times10^{-16}$ and $1.9\times10^{-15}$ for $d=1,\dots,5$; this checks the compatibility of the smoothness functionals themselves, which is more informative than checking $\curl\nabla=0$ after assembly.

\begin{table}[htbp]
\centering\scriptsize
\caption{Planar $H(\curl)$ manufactured-solution errors. Every degree line contains six independently computed mesh resolutions.}
\label{tab:flat-source}
\begin{tabular}{ccrrrr}
\toprule
$d$ & mesh & $\|e\|_{L^2}$ & rate & $\|\curl e\|_{L^2}$ & rate\\
\midrule
1 & 4 & $0.09091$ & -- & $0.5527$ & -- \\
1 & 6 & $0.0444$ & 1.77 & $0.4856$ & 0.32 \\
1 & 8 & $0.02553$ & 1.92 & $0.3762$ & 0.89 \\
1 & 12 & $0.01152$ & 1.96 & $0.2554$ & 0.96 \\
1 & 16 & $0.006511$ & 1.98 & $0.1926$ & 0.98 \\
1 & 24 & $0.002905$ & 1.99 & $0.1289$ & 0.99 \\
\addlinespace[2pt]
2 & 3 & $0.03319$ & -- & $0.2547$ & -- \\
2 & 4 & $0.01432$ & 2.92 & $0.1814$ & 1.18 \\
2 & 6 & $0.004393$ & 2.91 & $0.07217$ & 2.27 \\
2 & 8 & $0.001868$ & 2.97 & $0.04074$ & 1.99 \\
2 & 12 & $5.56\times10^{-4}$ & 2.99 & $0.01824$ & 1.98 \\
2 & 16 & $2.35\times10^{-4}$ & 2.99 & $0.01029$ & 1.99 \\
\addlinespace[2pt]
3 & 2 & $0.02411$ & -- & $0.08162$ & -- \\
3 & 3 & $0.005669$ & 3.57 & $0.06648$ & 0.51 \\
3 & 4 & $0.001868$ & 3.86 & $0.01968$ & 4.23 \\
3 & 6 & $3.79\times10^{-4}$ & 3.93 & $0.008277$ & 2.14 \\
3 & 8 & $1.21\times10^{-4}$ & 3.97 & $0.003619$ & 2.88 \\
3 & 12 & $2.41\times10^{-5}$ & 3.98 & $0.001094$ & 2.95 \\
\bottomrule
\end{tabular}
\end{table}

\begin{figure}[htbp]
\centering
\includegraphics[width=.48\textwidth]{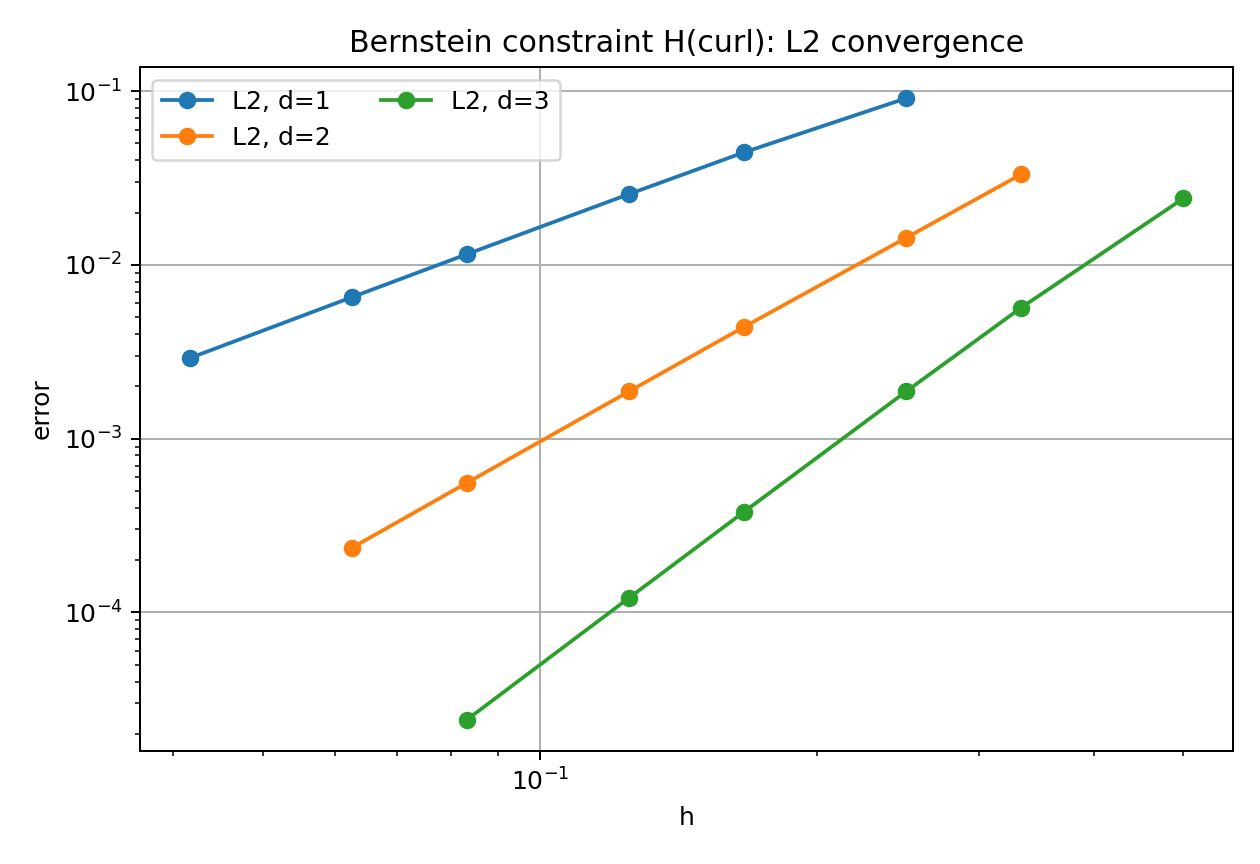}\hfill
\includegraphics[width=.48\textwidth]{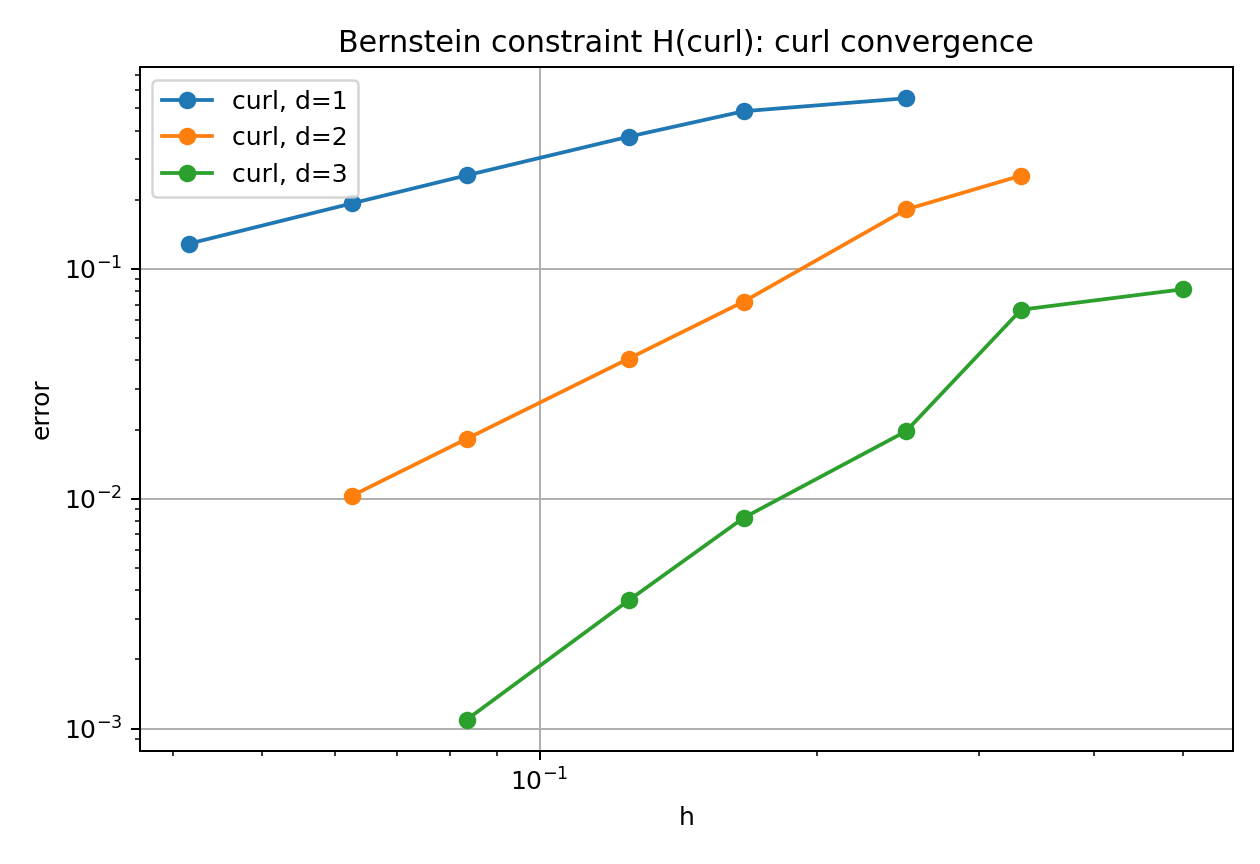}
\caption{Planar tangentially continuous source problem: $L^2$ and curl errors against mesh size for $d=1,2,3$. The asymptotic orders are $d+1$ and $d$.}
\label{fig:flat-source}
\end{figure}

\subsection{Planar normally continuous source problem}

With $U=RE$, the normal boundary condition for $U$ corresponds to the tangential one for $E$. We solve $(U_h,V_h)+(\divv U_h,\divv V_h)=\ell(V_h)$ with the load manufactured from the exact $U$, using $[\Pp_d]^2$ on every triangle and the normal functionals instead of the tangential ones. Table~\ref{tab:flat-hdiv} and Figure~\ref{fig:hdiv} show convergence identical, as rotation predicts, to the tangential experiment: the difference between the two spaces lies entirely in the edge block.

\begin{table}[htbp]
\centering\scriptsize
\caption{Planar $H(\divv)$ manufactured-solution errors using normal-trace constraints. Every degree line contains six independently computed mesh resolutions.}
\label{tab:flat-hdiv}
\begin{tabular}{ccrrrr}
\toprule
$d$ & mesh & $\|e\|_{L^2}$ & rate & $\|\divv e\|_{L^2}$ & rate\\
\midrule
1 & 4 & $0.09072$ & -- & $0.5527$ & -- \\
1 & 6 & $0.04429$ & 1.77 & $0.4856$ & 0.32 \\
1 & 8 & $0.02546$ & 1.92 & $0.3762$ & 0.89 \\
1 & 12 & $0.01148$ & 1.96 & $0.2554$ & 0.96 \\
1 & 16 & $0.006492$ & 1.98 & $0.1926$ & 0.98 \\
1 & 24 & $0.002896$ & 1.99 & $0.1289$ & 0.99 \\
\addlinespace[2pt]
2 & 3 & $0.03315$ & -- & $0.2547$ & -- \\
2 & 4 & $0.0143$ & 2.92 & $0.1814$ & 1.18 \\
2 & 6 & $0.004391$ & 2.91 & $0.07217$ & 2.27 \\
2 & 8 & $0.001868$ & 2.97 & $0.04074$ & 1.99 \\
2 & 12 & $5.56\times10^{-4}$ & 2.99 & $0.01824$ & 1.98 \\
2 & 16 & $2.35\times10^{-4}$ & 2.99 & $0.01029$ & 1.99 \\
\addlinespace[2pt]
3 & 2 & $0.02409$ & -- & $0.08162$ & -- \\
3 & 3 & $0.005666$ & 3.57 & $0.06648$ & 0.51 \\
3 & 4 & $0.001867$ & 3.86 & $0.01968$ & 4.23 \\
3 & 6 & $3.79\times10^{-4}$ & 3.93 & $0.008277$ & 2.14 \\
3 & 8 & $1.21\times10^{-4}$ & 3.96 & $0.003619$ & 2.88 \\
3 & 12 & $2.41\times10^{-5}$ & 3.98 & $0.001094$ & 2.95 \\
\bottomrule
\end{tabular}
\end{table}

\begin{figure}[htbp]
\centering
\includegraphics[width=.56\textwidth]{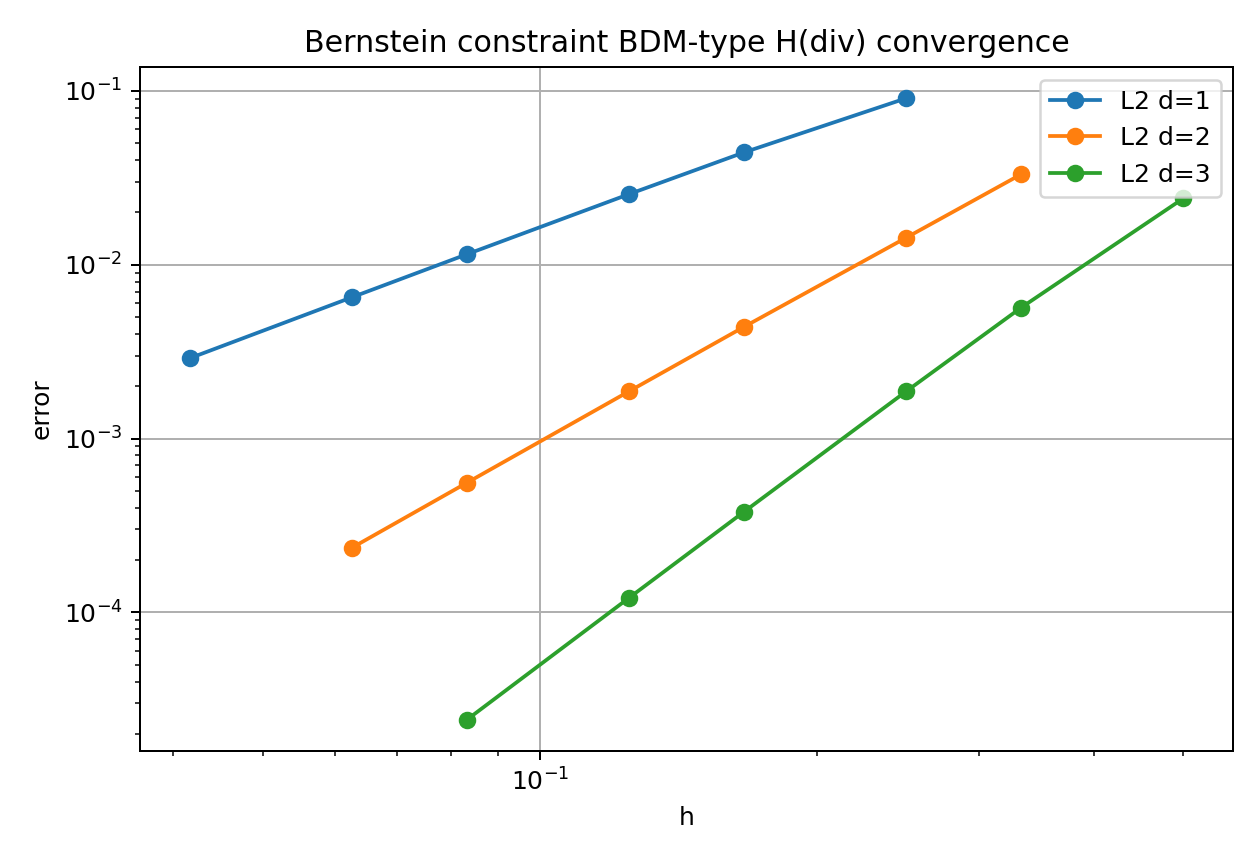}
\caption{Normally continuous (BDM-type) source problem. Only the normal B-coefficients are matched across edges.}
\label{fig:hdiv}
\end{figure}

\subsection{The Powell--Sabin complex and a \texorpdfstring{$C^1$}{C1} biharmonic problem}\label{sec:ps-numerics}

Each square of an $n\times n$ grid is split into two triangles and each triangle receives the six-split of Section~\ref{sec:ps-complex}. The code stores broken quadratic coefficients and the edge functionals \eqref{eq:ps-normal-derivative} and \eqref{eq:ps-J2}. An orthonormal basis of $\ker J_0$ is computed only as an offline diagnostic for the tests \eqref{eq:ps-tests}, never in a solve. Table~\ref{tab:ps-exactness} reports the ranks. On all six meshes the gap $\dim\ker(\mathsf D_{\rm PS}|_{\ker J_1})-\rank(\mathsf R_{\rm PS}|_{\ker J_0})$ is zero and $\dim\ker J_2=\rank(\mathsf D_{\rm PS}|_{\ker J_1})$; $\mathsf D_{\rm PS}\mathsf R_{\rm PS}=0$ holds at roundoff, and the two compatibility residuals stay below $6.7\times10^{-14}$ and $2.7\times10^{-15}$. At $n=3$, for example, the refinement has $108$ subtriangles, $\dim S^1_2=48$, $\dim L^1_1=134$, $\dim V^2_0=87$, and $\rank\rot=47=\dim S^1_2-1=\dim\ker\divv$. These are the rank identities of the exact sequence, obtained from the smoothness matrices.

\begin{table}[htbp]
\centering\scriptsize
\caption{Rank identities of the lowest-order Powell--Sabin exact complex, computed from the smoothness matrices, on six successively refined unit-square macrotriangulations. Here $r_{01}=\|J_1\mathsf R_{\rm ps}(I-J_0^+J_0)\|$ and $r_{12}=\|J_2\mathsf D_{\rm ps}(I-J_1^+J_1)\|$ are normalized in the code by the corresponding operator norms. Exactness is independently checked by the zero gap $\dim\ker(\mathsf D_{\rm ps}|_{\ker J_1})-\rank(\mathsf R_{\rm ps}|_{\ker J_0})$ and by $\dim\ker J_2=\rank(\mathsf D_{\rm ps}|_{\ker J_1})$.}
\label{tab:ps-exactness}
\setlength{\tabcolsep}{3pt}
\begin{tabular}{rrrrrrrrrr}
\toprule
$n$ & PS tri. & $\dim S^1_2$ & $\dim L^1_1$ & $\dim\ker J_2$ & $\rank\mathsf R_{\rm ps}$ & $\rank\mathsf D_{\rm ps}$ & gap & $r_{01}$ & $r_{12}$\\
\midrule
1 & 12 & 12 & 22 & 11 & 11 & 11 & 0 & $3.94\times10^{-15}$ & $5.36\times10^{-16}$ \\
2 & 48 & 27 & 66 & 40 & 26 & 40 & 0 & $1.02\times10^{-14}$ & $9.69\times10^{-16}$ \\
3 & 108 & 48 & 134 & 87 & 47 & 87 & 0 & $2.01\times10^{-14}$ & $1.59\times10^{-15}$ \\
4 & 192 & 75 & 226 & 152 & 74 & 152 & 0 & $3.42\times10^{-14}$ & $1.9\times10^{-15}$ \\
5 & 300 & 108 & 342 & 235 & 107 & 235 & 0 & $5.33\times10^{-14}$ & $2.27\times10^{-15}$ \\
6 & 432 & 147 & 482 & 336 & 146 & 336 & 0 & $6.69\times10^{-14}$ & $2.6\times10^{-15}$ \\
\bottomrule
\end{tabular}
\end{table}

To use the first space as a trial space, consider the clamped biharmonic problem $\Delta^2u=f$ on $(0,1)^2$ with $u=\partial_nu=0$ and exact solution $u=[x(1-x)y(1-y)]^2$. The energy $(D^2u_h,D^2v_h)$ is assembled on the broken quadratic coefficients, interior $C^1$ smoothness and the two clamped boundary conditions are imposed as exact functionals, and redundant rows are compressed by pivoted QR. Table~\ref{tab:ps-biharmonic} shows the $H^2$ seminorm rate settling to one and the $H^1$ and $L^2$ rates to two, as expected for the quadratic $C^1$ space, with full uncompressed residuals between $4.9\times10^{-16}$ and $2.6\times10^{-14}$. The same machinery that realizes the de Rham functionals thus solves a fourth-order problem in a smooth spline space.

\begin{table}[htbp]
\centering\scriptsize
\caption{Clamped biharmonic problem in the $C^1$ quadratic Powell--Sabin space at six nontrivial mesh resolutions. The exact solution is $u=[x(1-x)y(1-y)]^2$.}
\label{tab:ps-biharmonic}
\begin{tabular}{rrrrrrrrr}
\toprule
$n$ & PS tri. & conf. dim. & $L^2$ & rate & $H^1$ semi. & rate & $H^2$ semi. & rate\\
\midrule
2 & 48 & 3 & $0.001113$ & -- & $0.005575$ & -- & $0.04746$ & -- \\
3 & 108 & 12 & $5.04\times10^{-4}$ & 1.95 & $0.002558$ & 1.92 & $0.03172$ & 0.99 \\
4 & 192 & 27 & $2.98\times10^{-4}$ & 1.82 & $0.001517$ & 1.82 & $0.02432$ & 0.92 \\
5 & 300 & 48 & $1.95\times10^{-4}$ & 1.90 & $9.95\times10^{-4}$ & 1.89 & $0.01962$ & 0.96 \\
6 & 432 & 75 & $1.37\times10^{-4}$ & 1.95 & $6.99\times10^{-4}$ & 1.94 & $0.0164$ & 0.98 \\
8 & 768 & 147 & $7.75\times10^{-5}$ & 1.98 & $3.96\times10^{-4}$ & 1.97 & $0.01232$ & 0.99 \\
\bottomrule
\end{tabular}
\end{table}

\subsection{Three realizations on the same problems}\label{sec:three-realizations-numerical}

We now separate the choice of realization from the approximation. All three methods start from the same broken element matrices, quadrature, mesh, degree, load and raw smoothness equations. The direct saddle solve uses a row basis only to remove null multiplier coordinates. The null-space method builds $Z$ by sparse elimination and solves \eqref{eq:Zreduced}. The augmented Lagrangian route \eqref{eq:alw} keeps the full raw matrix in $C^TC$, starts from $\lambda^{(0)}=0$, factors $E_\eps$ once for each $\eps=10^j$, $j=-8,\dots,8$, and reuses the factorization. We record $k_u$, the first iterate with $\norm{c^{(m)}-c_{\rm sad}}/\norm{c_{\rm sad}}\le10^{-4}$, with $c_{\rm sad}$ the solution of the saddle system,, and $k_C$, the first with $\norm{Cc^{(m)}}/\norm{c^{(m)}}\le10^{-8}$; runs are capped at $1000$ updates. The condition numbers $\widehat\kappa_1$ are estimates of the one-norm condition number of the factored matrix, $E_\eps$ for the augmented method and the bordered matrix for the saddle solve, computed by Hager's one-norm estimator with inverse actions taken from the same sparse factorization used in the solve.

\paragraph{Classical de Rham case, $d=3$.}
The tangentially continuous problem of Section~\ref{sec:flat-hcurl} on the $8\times8$ mesh at $d=3$ has $2560$ broken coefficients, $832$ tangential functionals, and $1728$ unknowns after elimination. The saddle solve and the null-space method agree to $3.1\times10^{-12}$ in relative coefficient norm, both with $\norm{u-u_h}_{L^2}=1.22\times10^{-4}$ and $\norm{\curl(u-u_h)}_{L^2}=3.70\times10^{-3}$, and the reduced complex of Proposition~\ref{prop:z-complex} satisfies $\norm{C_1D_0Z_0}/\norm{D_0Z_0}=8.4\times10^{-17}$ and $\norm{\widehat D_1\widehat D_0}/(\norm{\widehat D_1}\norm{\widehat D_0})=4.8\times10^{-18}$. In the augmented sweep one outer solve reaches the $10^{-4}$ criterion for $10^{-7}\le\eps\le10^{-4}$, three suffice at $\eps=1$, and the count grows to $197$ at $\eps=10^3$ and beyond the cap for $\eps\ge10^4$; at $\eps=10^{-8}$ the joining residual is tiny while the coefficients stay away from the constrained solution, so a small joining residual alone is no evidence that the augmented solve is accurate. The operator $\curl\curl-I$ is indefinite and outside the hypothesis of the convergence theorem of \cite{AwanouLaiWenston2006}, so these counts are empirical.

\paragraph{Smooth case, $C^1$ and $d=5$.}
The second study uses the profile $S^1_5\to[S^0_4]^2\to S^{-1}_3$ on the same mesh, with the Poisson problem $-\Delta u=f$, $u=0$ on $\partial\Omega$, exact solution $\sin(\pi x)\sin(\pi y)$, and the approximation in $S^1_5(\tri)$. The smoothness matrix contains the conditions \eqref{eq:LS-smoothness} with $r=1$ on every interior edge and the Dirichlet rows: $2688$ broken coefficients and $2128$ raw equations of rank $1897$, hence $791$ unknowns. Both singularities of Section~\ref{sec:variational} are present, the broken stiffness matrix annihilating constants on each triangle and the raw $C^1$ equations carrying the vertex dependencies. The saddle solve and the null-space method agree to $1.1\times10^{-14}$, with $L^2$ error $7.66\times10^{-8}$ and $H^1$ seminorm error $6.06\times10^{-6}$, and the reduced complex satisfies $\norm{C_1D_0Z_0}/\norm{D_0Z_0}=3.4\times10^{-17}$ and $\norm{\widehat D_1\widehat D_0}/(\norm{\widehat D_1}\norm{\widehat D_0})=1.7\times10^{-18}$: the three realizations remain available when continuity is a derivative relation rather than an identification of coefficients. In this semidefinite case, covered by \cite[Thm.~6]{AwanouLaiWenston2006}, one outer solve suffices for $\eps\le10^{-4}$, nine are needed at $\eps=1$, $561$ at $\eps=10^2$, and the target is not reached for $\eps\ge10^3$, while the condition estimate falls from $6.2\times10^{11}$ at $\eps=10^{-8}$ to $2.9\times10^3$ near $\eps=10$ and rises again beyond; the parameter that minimizes the condition number does not minimize the outer count. Tables~\ref{tab:three-exact} and~\ref{tab:alw-two-sweeps} and Figures~\ref{fig:three-realizations-cond} and~\ref{fig:three-realizations-it} collect the sweeps.

\begin{table}[htbp]
\centering
\caption{The two parameter-free realizations in the comparison studies. The null-space matrix $Z$ is produced by sparse elimination. In the $C^1$ case the size and condition estimate of the saddle system refer to the row-compressed matrix used for factorization, while the residual is evaluated with the full raw smoothness matrix. $\widehat\kappa_1$ is the estimated one-norm condition number of the factored matrix.}
\label{tab:three-exact}
\footnotesize
\begin{tabular}{llrrrrrr}
\toprule
study & method & size & $\widehat\kappa_1$ & setup [s] & factor [s] & solve [s] & residual\\
\midrule
$C^0,d=3$ & saddle system & 3392 & 3.91e+05 & 0 & 0.00759 & 0.000447 & 2.11e-17\\
$C^0,d=3$ & null-space method & 1728 & 4.95e+05 & 0.141 & 0.0111 & 0.000528 & 1.65e-16\\
$C^1,d=5$ & saddle system & 4585 & 2.40e+03 & 0.108 & 0.0167 & 0.000959 & 1.84e-16\\
$C^1,d=5$ & null-space method & 791 & 4.93e+03 & 0.106 & 0.00512 & 0.000188 & 1.39e-16\\
\bottomrule
\end{tabular}
\end{table}
\begin{table}[htbp]
\centering
\caption{Augmented Lagrangian sweep on the problems of Table~\ref{tab:three-exact}. $k_u$ is the first outer iterate within $10^{-4}$ of the saddle-system solution; $k_C$ is the first with normalized constraint residual below $10^{-8}$.  ``$>1000$'' denotes failure to reach the criterion within the cap; $\widehat\kappa_1$ is the estimated one-norm condition number of $E_\eps$.}
\label{tab:alw-two-sweeps}
\scriptsize
\setlength{\tabcolsep}{3.5pt}
\begin{tabular}{r|rrr|rrr}
\toprule
& \multicolumn{3}{c|}{$C^0,d=3$: $H(\curl)$} & \multicolumn{3}{c}{$C^1,d=5$: Poisson}\\
$\log_{10}\eps$ & $\widehat\kappa_1$ & $k_u$ & $k_C$ & $\widehat\kappa_1$ & $k_u$ & $k_C$\\
\midrule
-8 & 5.20e+13 & $>1000$ & 1 & 6.22e+11 & 1 & 1\\
-7 & 5.23e+12 & 1 & 1 & 6.22e+10 & 1 & 1\\
-6 & 4.80e+11 & 1 & 2 & 6.22e+09 & 1 & 2\\
-4 & 5.24e+09 & 1 & 2 & 6.22e+07 & 1 & 2\\
-2 & 5.35e+07 & 2 & 3 & 6.24e+05 & 2 & 4\\
0 & 5.35e+05 & 3 & 5 & 7.92e+03 & 9 & 20\\
1 & 2.95e+05 & 6 & 10 & 2.90e+03 & 59 & 139\\
2 & 2.61e+05 & 23 & 47 & 1.08e+04 & 561 & $>1000$\\
3 & 2.66e+05 & 197 & 402 & 9.45e+04 & $>1000$ & $>1000$\\
4 & 2.74e+05 & $>1000$ & $>1000$ & 9.32e+05 & $>1000$ & $>1000$\\
6 & 2.37e+05 & $>1000$ & $>1000$ & 9.31e+07 & $>1000$ & $>1000$\\
8 & 2.37e+05 & $>1000$ & $>1000$ & 9.31e+09 & $>1000$ & $>1000$\\
\bottomrule
\end{tabular}
\end{table}

\begin{figure}[htbp]
\centering
\includegraphics[width=.66\textwidth]{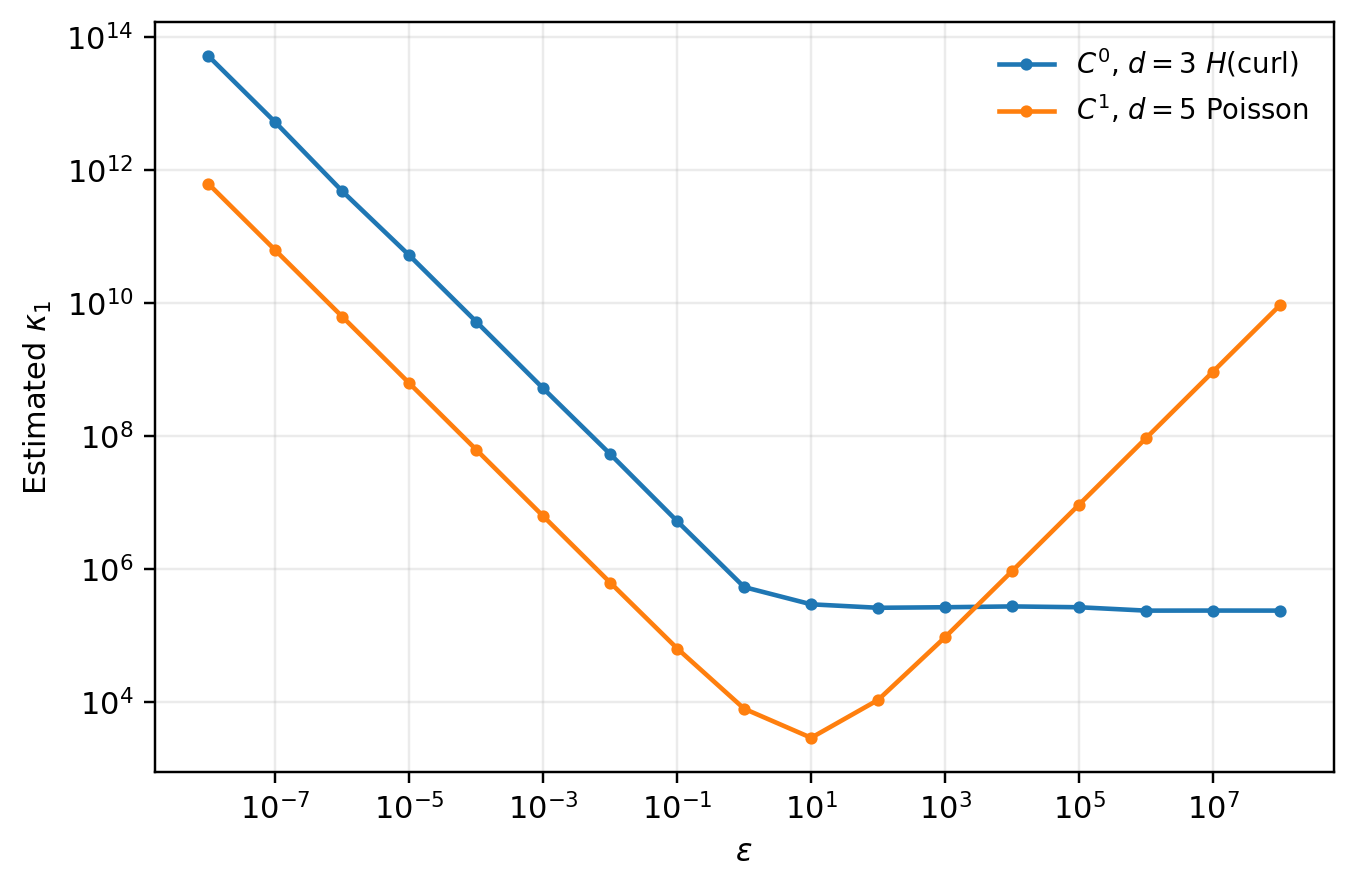}
\caption{Estimated one-norm condition number of the augmented matrix $E_\eps$ over the parameter sweep for the $d=3$ de Rham problem and the $C^1$, $d=5$ problem. Small $\eps$ enforces the constraints in one solve at the price of penalty conditioning; large $\eps$ exposes the broken operator, and for the semidefinite Poisson matrix the condition number rises again.}
\label{fig:three-realizations-cond}
\end{figure}

\begin{figure}[htbp]
\centering
\includegraphics[width=.66\textwidth]{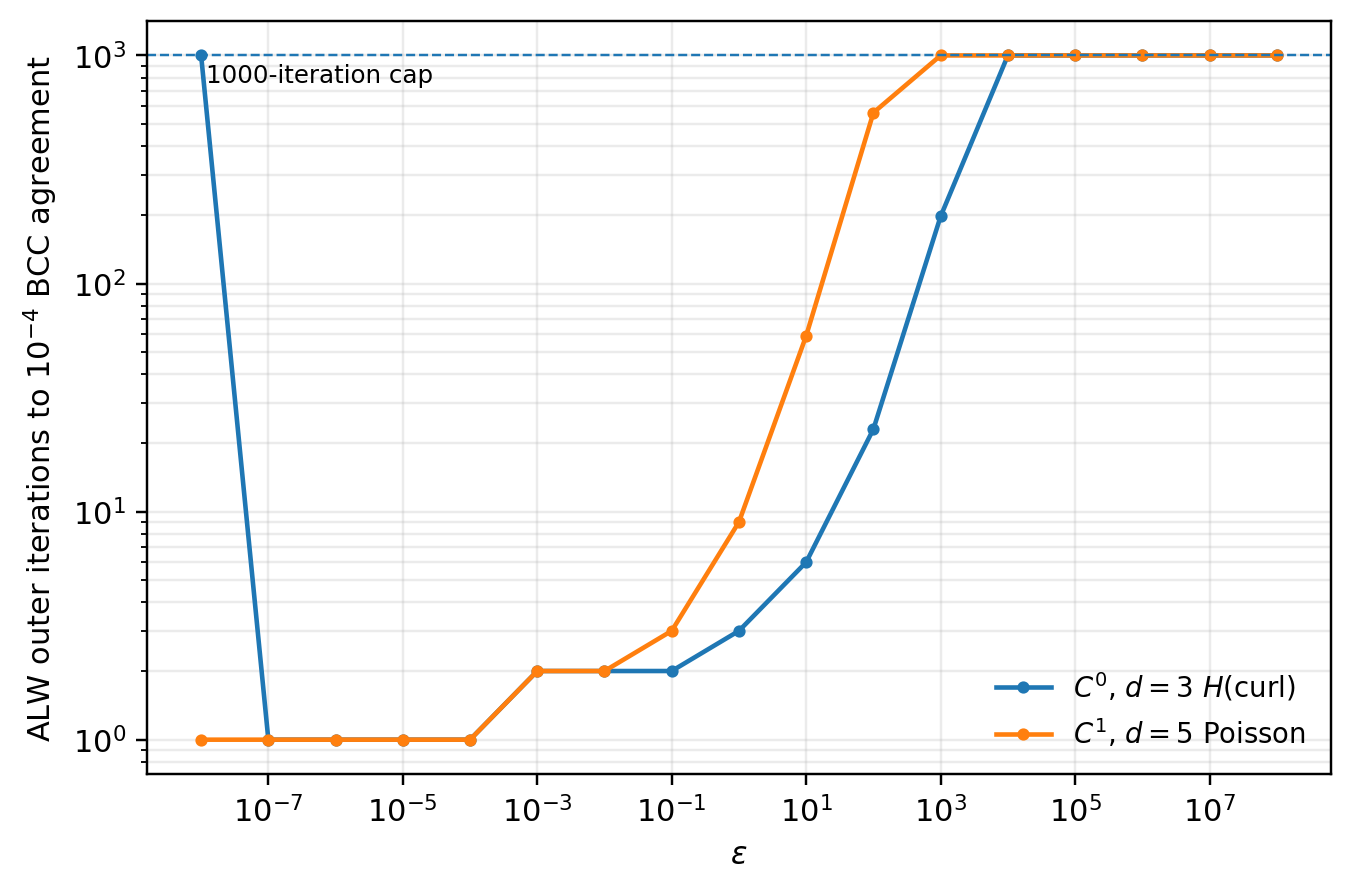}
\caption{Outer iterations of the augmented Lagrangian method to reach $10^{-4}$ relative agreement with the constrained solution from the saddle solve. Points at the cap did not meet the criterion within $1000$ updates.}
\label{fig:three-realizations-it}
\end{figure}

\paragraph{Assessment.}
The null-space method gives the smallest system, preserves the complex to roundoff, and is the method of choice when the null-space matrix is reused; in the $d=3$ study its determining set consists of the shared domain points, so its row in Table~\ref{tab:three-exact} is exactly the standard conforming assembly with a global basis, and the saddle solve's extra cost relative to it is the cost of the constrained representation; in the $C^1$ study it factors in a third of the time of the bordered system. The augmented iteration recovers the constrained solution in a few outer solves for a favorable parameter, and its behavior depends on the scaling and redundancy of the rows in $C^TC$ as well as on the kernel. The direct saddle solve needs neither a null-space matrix in every slot nor a parameter, which is the property the eigenvalue and surface experiments use, and it is the realization used for the remainder of the section.

\subsection{Planar Maxwell eigenvalues}

We solve the eigenproblem through the pencil of Theorem~\ref{thm:eigen-pencil} with a shift near $\pi^2$. Table~\ref{tab:flat-eigs} and Figure~\ref{fig:flat-eigs} show that the double eigenvalue $\pi^2$ is preserved at all six resolutions for every degree; at $d=3$ on the finest mesh the first two eigenvalues are $9.869604419$ and the third is $19.739209327$, against $9.869604401$ and $19.739208802$. The error lines are consistent with the $O(h^{2d})$ estimate.

\begin{table}[htbp]
\centering\scriptsize
\caption{First three positive Maxwell eigenvalues on the unit square. The exact values are $\pi^2,\pi^2,2\pi^2$. Six mesh resolutions are reported for each degree.}
\label{tab:flat-eigs}
\begin{tabular}{ccrrr}
\toprule
$d$&mesh&$\lambda_{h,1}$&$\lambda_{h,2}$&$\lambda_{h,3}$\\
\midrule
1 & 4 & 10.208127030 & 10.208127030 & 21.063702517 \\
1 & 6 & 10.020066262 & 10.020066262 & 20.335873620 \\
1 & 8 & 9.954213523 & 9.954213523 & 20.076074309 \\
1 & 12 & 9.907196652 & 9.907196652 & 19.889276980 \\
1 & 16 & 9.890747322 & 9.890747322 & 19.823686341 \\
1 & 24 & 9.879000336 & 9.879000336 & 19.776774091 \\
\addlinespace[2pt]
2 & 3 & 9.879001651 & 9.879001651 & 19.826392231 \\
2 & 4 & 9.872643160 & 9.872643160 & 19.768306182 \\
2 & 6 & 9.870214414 & 9.870214414 & 19.745189072 \\
2 & 8 & 9.869798530 & 9.869798530 & 19.741128299 \\
2 & 12 & 9.869642908 & 9.869642908 & 19.739591924 \\
2 & 16 & 9.869616603 & 9.869616603 & 19.739330470 \\
\addlinespace[2pt]
3 & 2 & 9.870380391 & 9.870380391 & 19.775381942 \\
3 & 3 & 9.869675648 & 9.869675648 & 19.741160016 \\
3 & 4 & 9.869617279 & 9.869617279 & 19.739572303 \\
3 & 6 & 9.869605544 & 9.869605544 & 19.739241764 \\
3 & 8 & 9.869604605 & 9.869604605 & 19.739214735 \\
3 & 12 & 9.869604419 & 9.869604419 & 19.739209327 \\
\bottomrule
\end{tabular}
\end{table}

\begin{figure}[htbp]
\centering
\includegraphics[width=.58\textwidth]{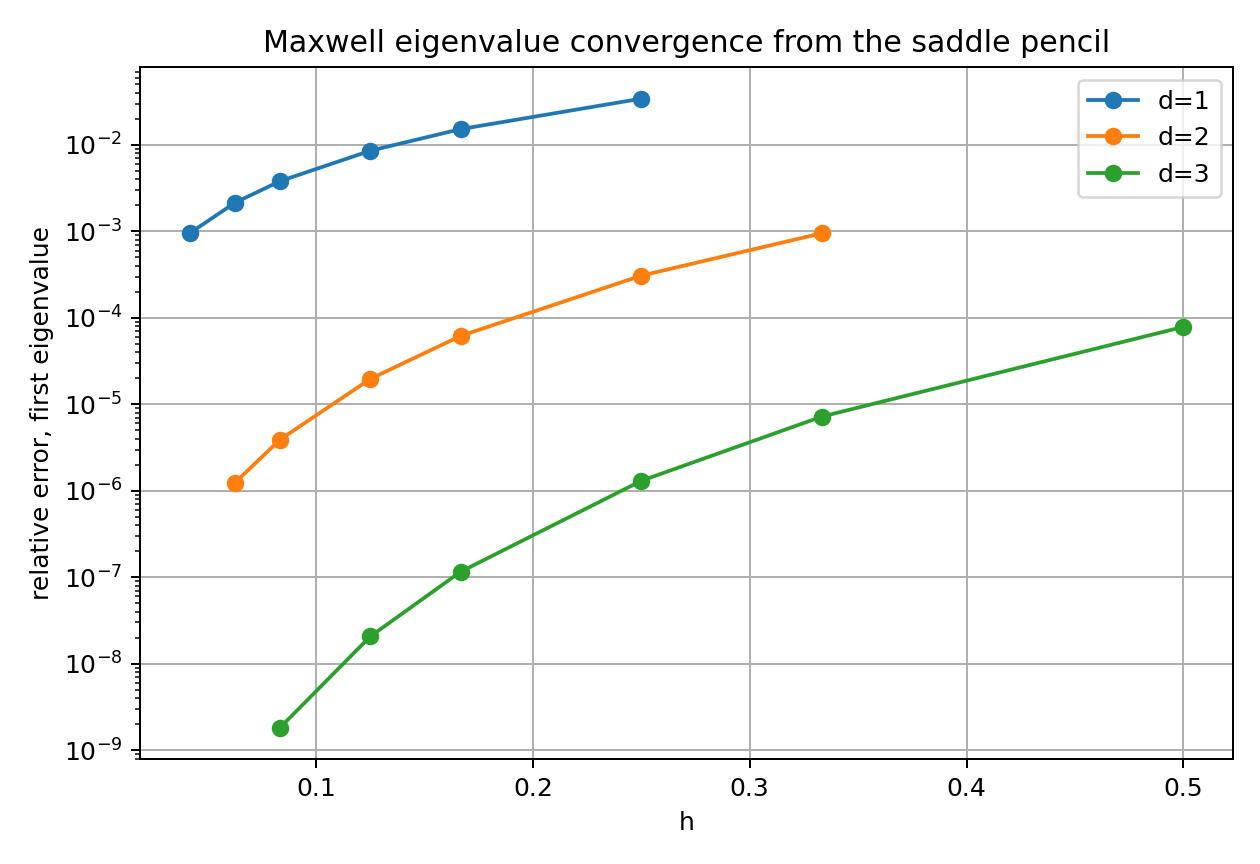}
\caption{Relative error in the first positive Maxwell eigenvalue on the unit square, computed from the constrained pencil without a basis of $\ker C_1^{\curl}$.}
\label{fig:flat-eigs}
\end{figure}

\subsection{The exact sphere}

Starting from an icosahedron, each face is subdivided with geodesic frequency $q=1,\dots,6$, the lattice vertices are normalized to the unit sphere, and every subtriangle is mapped by \eqref{eq:sphere-map}, giving $20q^2$ curved triangles that tile the sphere exactly. The tangential field $u=n\times e_z=(y,-x,0)$ has $\curl_\Gamma u=-2z$, and we solve the shifted problem $(\curl_\Gamma u_h,\curl_\Gamma v_h)+(u_h,v_h)=(\curl_\Gamma u,\curl_\Gamma v_h)+(u,v_h)$ with the covariant map and the reference tangential functionals. Table~\ref{tab:sphere-source} shows optimal orders at every degree; at $d=3$ the $L^2$ error falls from $2.786\times10^{-2}$ to $2.354\times10^{-5}$ with final rate $3.99$ and the curl error from $2.329\times10^{-1}$ to $1.116\times10^{-3}$ with final rate $3.00$.

\begin{table}[htbp]
\centering\scriptsize
\caption{$H(\curl_\Gamma)$ source errors on the exactly represented unit sphere. For each degree $d=1,2,3$ the geodesic frequency $q=1,\ldots,6$ gives $20q^2$ curved triangles, hence six resolutions.}
\label{tab:sphere-source}
\begin{tabular}{ccrrrr}
\toprule
$d$&$q$&$\|e\|_{L^2}$&rate&$\|\curl_\Gamma e\|_{L^2}$&rate\\
\midrule
1 & 1 & $0.3757$ & -- & $1.535$ & -- \\
1 & 2 & $0.09318$ & 2.01 & $0.7345$ & 1.06 \\
1 & 3 & $0.04099$ & 2.03 & $0.4836$ & 1.03 \\
1 & 4 & $0.02305$ & 2.00 & $0.3613$ & 1.01 \\
1 & 5 & $0.01476$ & 2.00 & $0.2885$ & 1.01 \\
1 & 6 & $0.01025$ & 2.00 & $0.2402$ & 1.01 \\
\addlinespace[2pt]
2 & 1 & $0.1308$ & -- & $0.7591$ & -- \\
2 & 2 & $0.01648$ & 2.99 & $0.197$ & 1.95 \\
2 & 3 & $0.004818$ & 3.03 & $0.08737$ & 2.00 \\
2 & 4 & $0.002033$ & 3.00 & $0.04923$ & 1.99 \\
2 & 5 & $0.001041$ & 3.00 & $0.03153$ & 2.00 \\
2 & 6 & $6.02\times10^{-4}$ & 3.00 & $0.0219$ & 2.00 \\
\addlinespace[2pt]
3 & 1 & $0.02786$ & -- & $0.2329$ & -- \\
3 & 2 & $0.0019$ & 3.87 & $0.03028$ & 2.94 \\
3 & 3 & $3.71\times10^{-4}$ & 4.02 & $0.008891$ & 3.02 \\
3 & 4 & $1.19\times10^{-4}$ & 3.97 & $0.003761$ & 2.99 \\
3 & 5 & $4.87\times10^{-5}$ & 3.99 & $0.001927$ & 3.00 \\
3 & 6 & $2.35\times10^{-5}$ & 3.99 & $0.001116$ & 3.00 \\
\bottomrule
\end{tabular}
\end{table}

The contravariant branch is exercised with $w=\nabla_\Gamma z=e_z-zn$, $\divv_\Gamma w=-2z$, in the reaction--divergence problem with normal functionals and the contravariant map. For $d=1,2,3$ its errors agree with the corresponding rows of Table~\ref{tab:sphere-source} to the displayed digits after replacing curl by divergence, and the constraint residuals at $q=6$ are $9.8\times10^{-16}$, $6.3\times10^{-16}$ and $1.2\times10^{-15}$; this is the curved analogue of the planar rotation check and confirms Theorem~\ref{thm:surface-flux} directly. The first positive curl--curl eigenvalue on the sphere is $2$ with multiplicity three. Table~\ref{tab:sphere-eigs} and Figure~\ref{fig:sphere-results} show the multiplicity preserved at all six frequencies; at $d=3$ the eigenvalue moves from $2.006309985$ at $q=1$ to $2.000000148$ at $q=6$.

\begin{table}[htbp]
\centering\scriptsize
\caption{First spherical Maxwell/Hodge eigenvalue cluster. The exact eigenvalue is $2$ with multiplicity three. Six geodesic frequencies are reported for each degree.}
\label{tab:sphere-eigs}
\begin{tabular}{ccrrr}
\toprule
$d$&$q$&$\lambda_{h,1}$&$\lambda_{h,2}$&$\lambda_{h,3}$\\
\midrule
1 & 1 & 2.302696220 & 2.302696220 & 2.302696361 \\
1 & 2 & 2.065097274 & 2.065097274 & 2.065097274 \\
1 & 3 & 2.028030001 & 2.028030001 & 2.028030001 \\
1 & 4 & 2.015612699 & 2.015612699 & 2.015612699 \\
1 & 5 & 2.009947162 & 2.009947162 & 2.009947162 \\
1 & 6 & 2.006890952 & 2.006890952 & 2.006890952 \\
\addlinespace[2pt]
2 & 1 & 2.067725628 & 2.067725628 & 2.067725992 \\
2 & 2 & 2.004580149 & 2.004580149 & 2.004580149 \\
2 & 3 & 2.000906122 & 2.000906122 & 2.000906122 \\
2 & 4 & 2.000288344 & 2.000288344 & 2.000288344 \\
2 & 5 & 2.000118409 & 2.000118409 & 2.000118409 \\
2 & 6 & 2.000057185 & 2.000057185 & 2.000057185 \\
\addlinespace[2pt]
3 & 1 & 2.006309985 & 2.006310015 & 2.006310015 \\
3 & 2 & 2.000108562 & 2.000108562 & 2.000108562 \\
3 & 3 & 2.000009404 & 2.000009404 & 2.000009404 \\
3 & 4 & 2.000001685 & 2.000001685 & 2.000001685 \\
3 & 5 & 2.000000443 & 2.000000443 & 2.000000443 \\
3 & 6 & 2.000000148 & 2.000000148 & 2.000000148 \\
\bottomrule
\end{tabular}
\end{table}

\begin{figure}[htbp]
\centering
\includegraphics[width=.48\textwidth]{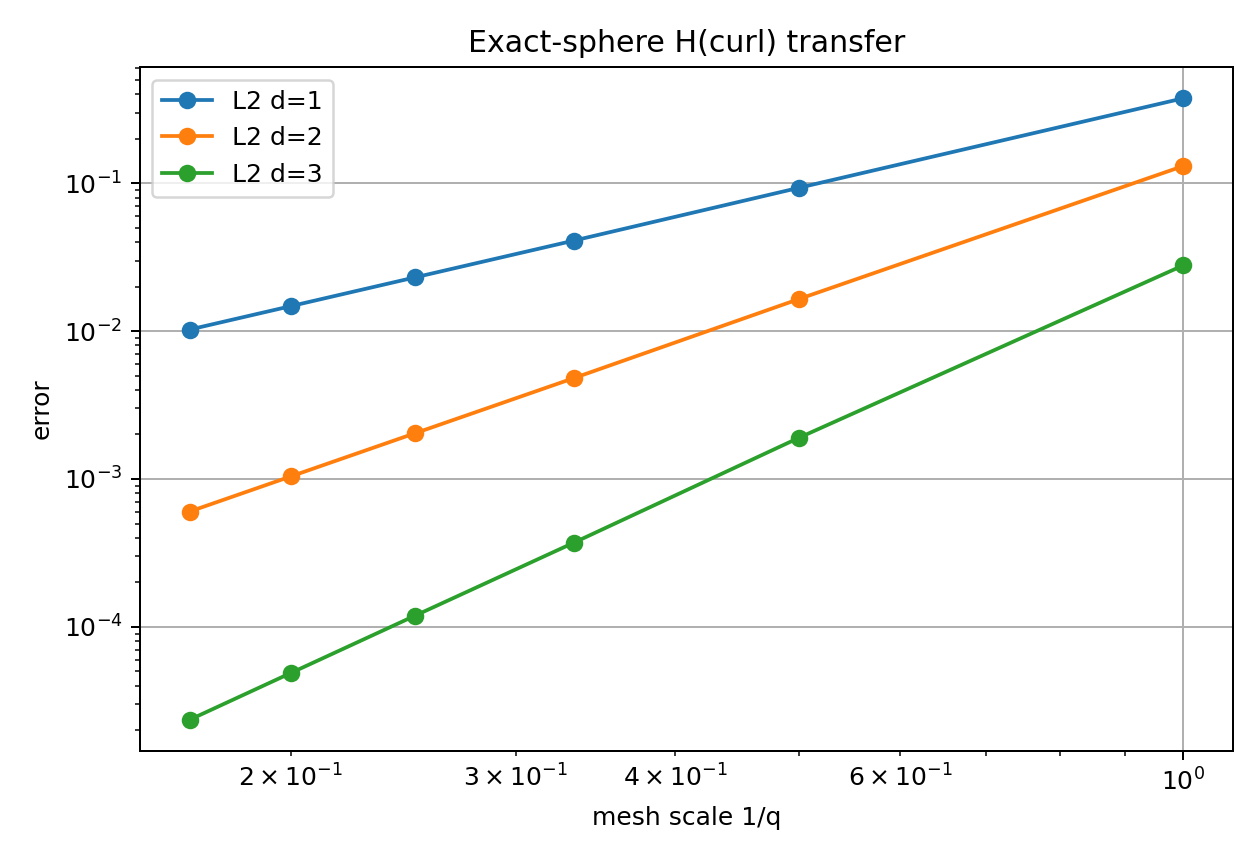}\hfill
\includegraphics[width=.48\textwidth]{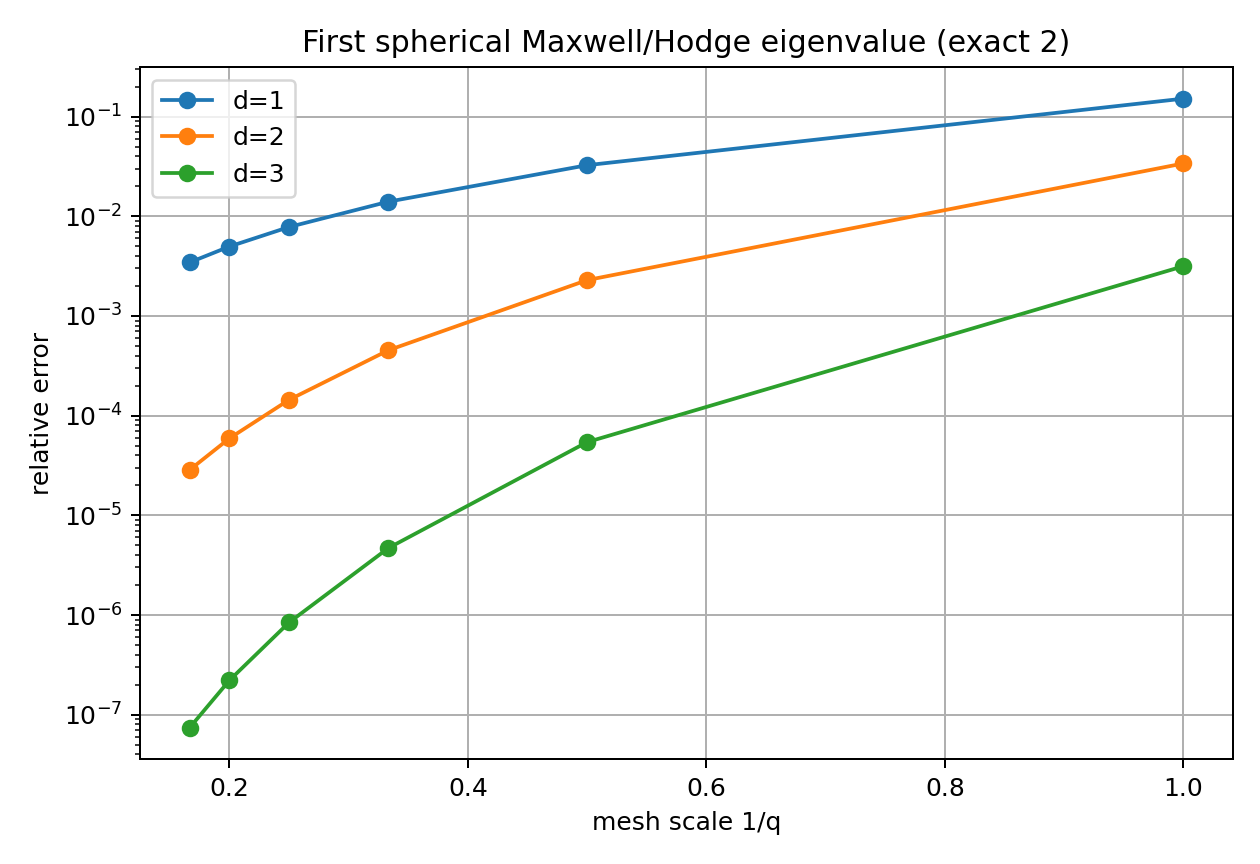}
\caption{Tangentially continuous approximation on the exact sphere. Left: $L^2$ convergence of the source problem. Right: relative error in the first coexact eigenvalue $2$.}
\label{fig:sphere-results}
\end{figure}

\subsection{The embedded hyperboloid}

On the hyperboloid patch the planar field \eqref{eq:manufactured-E} is mapped by the covariant Piola map of $\Phi$, $E_\Gamma=D\Phi\,G_\Phi^{-1}\hat E$ with $G_\Phi=(D\Phi)^TD\Phi$, so that $\curl_\Gamma E_\Gamma=\mathcal J_\Phi^{-1}\widehat\curl\hat E$, and the load is manufactured from these fields. Table~\ref{tab:hyper} and Figure~\ref{fig:hyper-result} show rates two in $L^2$ and one in surface curl at $d=1$, with a constraint residual that grows mildly with the geometric conditioning and is $1.2\times10^{-14}$ on the finest mesh.

\begin{table}[htbp]
\centering\small
\caption{Degree-one $H(\curl_\Gamma)$ source errors on the embedded hyperboloid patch at six mesh resolutions.}
\label{tab:hyper}
\begin{tabular}{crrrrr}
\toprule
mesh&$\|e\|_{L^2}$&rate&$\|\curl_\Gamma e\|_{L^2}$&rate&$\|Cu\|/\|u\|$\\
\midrule
4 & $0.09093$ & -- & $0.5125$ & -- & $4.5\times10^{-16}$ \\
6 & $0.04442$ & 1.77 & $0.4503$ & 0.32 & $5.21\times10^{-16}$ \\
8 & $0.0255$ & 1.93 & $0.3488$ & 0.89 & $1.72\times10^{-15}$ \\
12 & $0.01149$ & 1.97 & $0.2368$ & 0.96 & $1.98\times10^{-15}$ \\
16 & $0.006497$ & 1.98 & $0.1786$ & 0.98 & $3.8\times10^{-15}$ \\
24 & $0.002897$ & 1.99 & $0.1195$ & 0.99 & $1.18\times10^{-14}$ \\
\bottomrule
\end{tabular}
\end{table}

\begin{figure}[htbp]
\centering
\includegraphics[width=.58\textwidth]{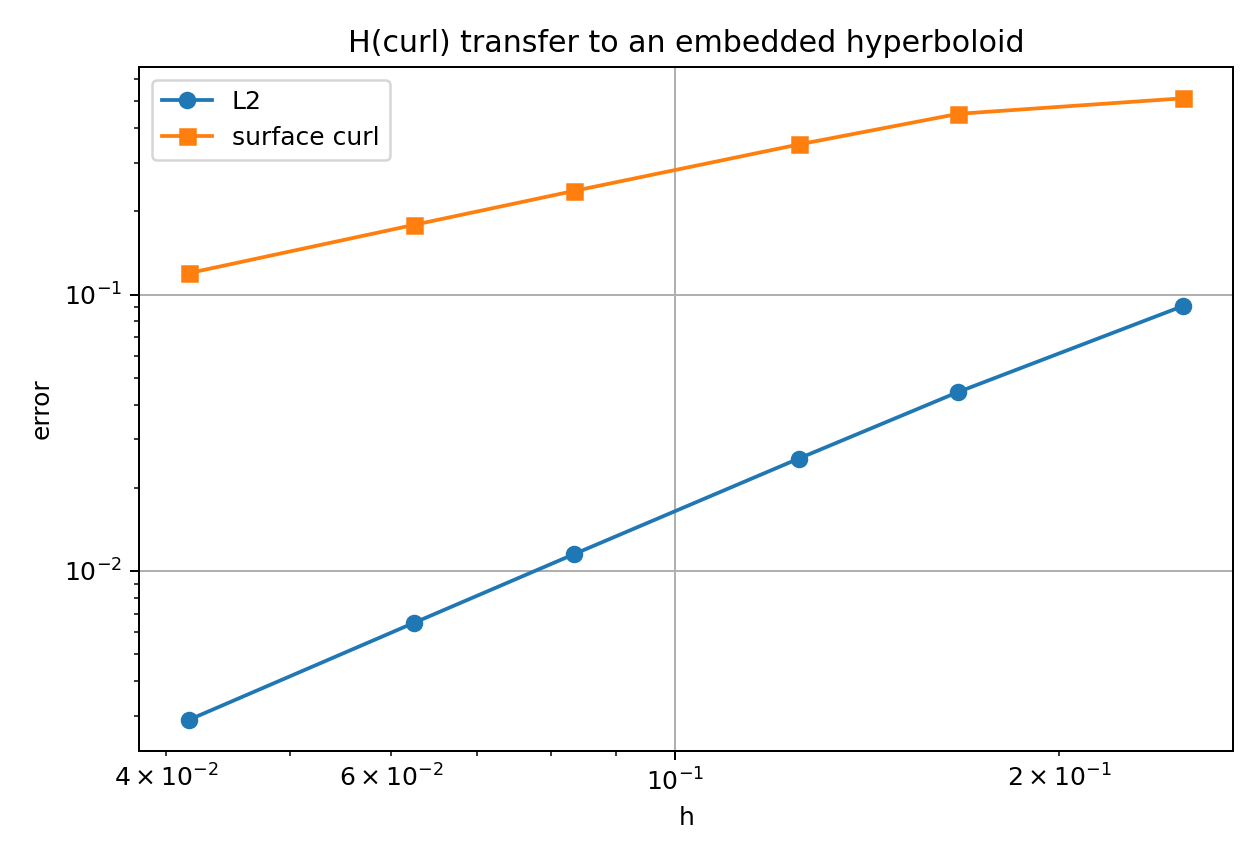}
\caption{Source convergence on the embedded hyperboloid patch at $d=1$. Curvature enters the element matrices only.}
\label{fig:hyper-result}
\end{figure}

\section{Discussion and conclusions}\label{sec:conclusion}

The four Sobolev spaces differ, at the level of implementation, only in the edge functional: coefficients at shared domain points for $H^1$, coefficients of the tangential component for $H(\curl)$, of the normal component for $H(\divv)$, and none for $L^2$. Higher smoothness adds derivative layers to the same edge object. Local assembly, kernel handling and interface preconditioning are therefore shared across scalar and vector problems, which is the situation in coupled problems such as electromagnetic scattering with smooth spline geometry, surface flows on closed surfaces, and plate or shell models that combine $C^1$ scalar unknowns with $H(\divv)$ fluxes, and the smoothness matrices keep $O(N_E)$ rows and nonzeros under refinement for fixed degree. The realizations differ in fill: $C^TC$ couples coefficients sharing a functional, $Z^TAZ$ reflects the support of the determining set, and the bordered system is a sparse saddle matrix. All timings here use direct factorizations and compare realizations; preconditioners uniform in mesh size, degree, smoothness order and curvature are the main open scalability question, and the interface operator $CA^+C^T$ with the coarse block $CR$ of Proposition~\ref{prop:pseudoinverse} is the natural starting point for them.

The row-space criterion is complementary to the theory of smooth exact sequences. A profile can be a subcomplex and still fail exactness through singular vertices or geometry-dependent dimensions, and finite element exterior calculus needs bounded commuting projections beyond the identity $D_{k+1}D_k=0$ \cite{ArnoldFalkWinther2010}. Theorem~\ref{thm:jet-commute} settles the subcomplex question for componentwise profiles in closed form, and the Powell--Sabin experiment shows that once an exact smooth sequence is known, all of its conditions can be stored as edge functionals and solved without its nodal basis. Maxwell eigenvalues are the natural benchmark for this because spurious modes expose defects in the gradient--curl structure that source problems hide; the preserved multiplicities on the square and the sphere, together with the optimal rates, validate the compatible representation.

Everything in this paper is two-dimensional and solved by direct factorization. The extension to tetrahedral meshes, where scalar traces live on faces, tangential continuity uses face and edge traces, normal continuity uses face traces, and Bernstein--B\'ezier bases respecting these operators exist \cite{AinsworthFu2018}, is future work, as is the transfer of $C^r$ smoothness functionals to curved patches. The obstruction on curved patches is that a transverse derivative of $u\circ F_T$ on a reference edge involves the derivative of the geometry map, so the reference functionals \eqref{eq:LS-smoothness} express $C^r$ continuity on $\Gamma_h$ only when the maps of the two triangles sharing an edge join with parametric $C^r$ continuity across it; for maps that join only continuously, the functionals must be replaced by geometry-dependent ones built from the derivatives of $F_T$ along the edge. Exact parameterizations with $C^r$-compatible element maps, such as the radial map of the sphere across edges of the same icosahedral face, are the natural first case. For exactness of general smooth profiles the obstructions are the singular vertices and near-singular configurations that make $\dim S^r_d(\tri)$ geometry dependent for $d<3r+2$, and the row-space test is the diagnostic that separates the algebraic subcomplex property from these questions. The remaining open question is how the interface systems should be preconditioned uniformly in mesh size, degree, smoothness and curvature; the interface operator $CA^+C^T$ with the coarse block $CR$ is the natural object for a block preconditioner.

\begin{appendices}

\section{The constrained shift-invert eigensolver}\label{app:eigs}

For a shift $\sigma$ that is not a finite eigenvalue, $S_\sigma=\mathcal A-\sigma\mathcal M=\begin{bmatrix}K-\sigma M&C^T\\C&0\end{bmatrix}$ is nonsingular. One application of the shift-invert operator to $x=(u,\mu)$ solves $S_\sigma y=\mathcal Mx=(Mu,0)$. If $x$ lies in the finite eigenspace of eigenvalue $\lambda$, then $y=(\lambda-\sigma)^{-1}x$ up to the generalized normalization, and every vector in $\ker\mathcal M=\{0\}\times\R^m$ is mapped to zero. Sparse Arnoldi iteration on $S_\sigma^{-1}\mathcal M$ therefore finds the finite constrained eigenvalues nearest $\sigma$. We take $\sigma\approx\pi^2$ on the square and $\sigma\approx2$ on the sphere. The zero eigenvalue of the discrete gradients is part of the curl--curl kernel and is not selected by these shifts.

\section{Reproducibility}\label{app:repro}

The code archive contains the following scripts.
\begin{itemize}[leftmargin=2em]
\item \texttt{bb\_complex.py}: Bernstein bases, gradient and curl matrices, edge functionals, planar source and eigenvalue solvers, and the compatibility test.
\item \texttt{hdiv\_complex.py}: the normally continuous source problem.
\item \texttt{curved\_complex.py}: surface maps, Piola assembly, source problems, and constrained eigenvalues.
\item \texttt{run\_flat.py}, \texttt{run\_hdiv.py}, \texttt{run\_curved.py}, \texttt{run\_curved\_hdiv.py}: the planar and curved validation data.
\item \texttt{run\_three\_realizations\_two\_studies.py}: the $d=3$ and $C^1$, $d=5$ comparison with the parameter sweeps and complex residuals.
\item \texttt{run\_maxwell\_kernel\_diagnostic.py}: the gradient and curl kernel check of Section~\ref{sec:implementation}.
\item \texttt{smooth\_powell\_sabin.py}: the incenter six-split, the Powell--Sabin functionals, the exactness diagnostics, and the biharmonic solve.
\end{itemize}
No library element and no global conforming basis is used in the solvers; the comparison script constructs the null-space matrix $Z$ by sparse elimination for the three-realization study only.

\end{appendices}

\section*{Declarations}

\bmhead{Funding}
Not applicable. No funding was received for this work.

\bmhead{Conflict of Interest}
The author declares that there is no conflict of interest.

\bmhead{Author Contributions}
Not applicable. The manuscript has a single author, who is responsible for all of the work.

\bmhead{Acknowledgements}
Not applicable.

\bmhead{Data and Code Availability}
The data underlying the tables and figures and the reference implementation are included with the manuscript source and are listed in Appendix~\ref{app:repro}.

\bmhead{Use of Generative AI}
The author used Claude Fable 5.1 (Anthropic) to assist with the organization of the draft, language editing, LaTeX preparation, and as a coding assistant for the reference implementation. The mathematical statements and their proofs, the numerical results, and the final text are the author's responsibility, and the author checked all sources cited.

\end{document}